\documentclass[11pt]{article}
\usepackage[]{amsmath,amssymb,amsfonts,latexsym,amsthm,enumerate}
\usepackage{mathrsfs}
\usepackage[numeric,initials,nobysame]{amsrefs}
\usepackage[bookmarks=false]{hyperref}

\def\given{\,|\,}

\def\S{\mathcal{S}}
\def\I{\mathcal{I}}
\def\RR{\mathbb{R}}
\def\F{\mathcal{F}}

\def\A{\mathcal{A}}
\def\B{\mathcal{B}}
\def\C{\mathcal{C}}
\def\D{\mathcal{D}}
\def\E{\mathcal{E}}
\def\EE{\mathfrak{E}}
\def\TT{\mathfrak{T}}
\def\floor#1{\lfloor #1 \rfloor}

\def\M{{M}}
\def\Trav{\text{{\it Traversed}}}
\def\NotTrav{\text{{\it NotTraversed}}}
\def\Labels{\text{{\it Labels}}}
\def\Roots{\text{{\it Roots}}}
\def\Children{\text{{\it Children}}}
\def\Triangles{\text{{\it Triangles}}}
\def\Edges{\text{{\it Edges}}}
\def\BIGBite{\text{{\it BIGBite}}}
\def\BigBite{\text{{\it BigBite}}}
\def\Bite{\text{{\it Bite}}}
\def\Pairs{\text{{\it Pairs}}}

\date{}
\title{The $K_4$-free process}
\author{{Guy Wolfovitz\thanks{Department of Computer Science, 
Haifa University, Haifa, Israel. Email address:
{\tt gwolfovi@cs.haifa.ac.il}.}}}

\newtheorem{theorem}{Theorem}[section]
\newtheorem{lemma}[theorem]{Lemma}
\newtheorem{claim}[theorem]{Claim}

\renewcommand{\epsilon}{\varepsilon}

\newcommand{\deq}{:=}

\DeclareMathOperator{\expec}{\mathbb{E}}
\DeclareMathOperator{\prob}{Pr}

\DeclareMathOperator{\ONE}{{\bf 1}}

\newtheoremstyle{upright}%
        {8pt plus2pt minus4pt}%
        {8pt plus2pt minus4pt}%
        {\upshape}%
        {}%
        {\bfseries}%
        {:}%
        {1em}%
        {}%
\theoremstyle{upright}

\newcommand{\ignore}[1]{}
\begin{document}

\maketitle

\begin{abstract}
We consider the $K_4$-free process. In this process, the edges of the complete
$n$-vertex graph are traversed in a uniformly random order, and each traversed
edge is added to an initially empty evolving graph, unless the addition of the
edge creates a copy of $K_4$.  Let $\M(n)$ denote the graph that is produced by
that process. We prove that a.a.s., the number of edges in $\M(n)$ is
$O(n^{8/5} (\ln n)^{1/5} )$. This matches, up to a constant factor, a lower
bound of Bohman.
As a by-product, we prove the following Ramsey-type result: for every $n$ there
exists a $K_4$-free $n$-vertex graph, in which the largest set of vertices that
doesn't span a triangle has size $O(n^{3/5}(\ln n)^{1/5})$. This improves, by a
factor of $(\ln n)^{3/10}$, an upper bound of Krivelevich.
\end{abstract}

\section{Introduction} \label{sec:1}
The $K_4$-free process is a random greedy process that generates a $K_4$-free
graph. In this process, the edges in $\binom{[n]}{2}$, where $[n] \deq \{1, 2,
\ldots, n\}$, are traversed in a uniformly random order, and each traversed
edge is added to an evolving graph, which is initially empty, unless the
addition of the edge creates a copy of $K_4$.  Denote by $\M(n)$ the (maximal)
$K_4$-free graph that is produced by that process.  Say that an event holds
asymptotically almost surely (a.a.s.) if the probability of that event goes to
$1$ as $n \to \infty$.  Throughout the paper we assume that $n \to \infty$, and
any asymptotic notation is used under this, and only this, assumption. The next
theorem is our main result.
\begin{theorem} \label{thm:main}
A.a.s., the number of edges in $\M(n)$ is $O( n^{8/5} (\ln n)^{1/5} )$.  
\end{theorem}

The first to study the $K_4$-free process were Bollob\'as and
Riordan~\cite{MR1756287}, who showed that a.a.s., the number of edges in
$\M(n)$ is lower bounded by $\Omega(n^{8/5})$ and upper bounded by $O(n^{8/5}
\ln n)$.  Improved results were provided by Osthus and Taraz~\cite{MR1799803},
who showed that a.a.s., the number of edges in $\M(n)$ is upper bounded by $O(
n^{8/5}(\ln n)^{1/2} )$, and by Bohman~\cite{MR2522430}, who showed that
a.a.s., the number of edges in $\M(n)$ is lower bounded by $\Omega(n^{8/5}(\ln
n)^{1/5})$.  Our main result shows that Bohman's lower bound is optimal up to
the hidden constant factor.

The $K_4$-free process is an instance of the more general $H$-free process,
where instead of forbidding the appearance of a copy of $K_4$ in the evolving
graph, one forbids the appearance of some fixed graph $H$. We will not discuss
the $H$-free process here (we refer the reader to~\cite{BKeevash}), but we will
say that the analysis of that process proves to be very useful in studying
certain problems in extremal combinatorics.  For example, the analysis of the
$H$-free process in~\cites{MR2522430, BKeevash} yielded the currently best
known lower bounds for the off-diagonal Ramsey numbers $R(k,n)$, for every
fixed $k \ge 4$. Our analysis of the $K_4$-free process yields a new result
regarding another Ramsey-type problem.  We briefly discuss that problem now.

Let $2 \le k < l \le n$ be integers. For a graph $G$, let $f_k(G)$ be the
maximum size of a subset of the vertices of $G$ that spans no copy of $K_k$.
Let $f_{k,l}(n)$ be defined to be $\min_G f_k(G)$, where the minimum is taken
over all $K_l$-free $n$-vertex graphs $G$.  For various choices of the
parameters, the function $f_{k,l}(n)$ was studied in~\cites{MR0144332,
MR0141612, MR1469821, MR1091586, MR1369060, MR1300971, MR2127369}. For the
special case where $k=3$ and $l=4$, Krivelevich~\cites{MR1300971, MR1369060}
showed that $f_{3,4}(n) = \Omega((n\ln\ln n)^{1/2})$ and $f_{3,4}(n) =
O(n^{3/5}(\ln n)^{1/2})$. The proof of Theorem~\ref{thm:main}, as we argue
below, gives as a by-product the following result, which improves Krivelevich's
upper bound by a factor of $(\ln n)^{3/10}$.
\begin{theorem} \label{thm:ramsey}
$f_{3,4}(n) = O(n^{3/5} (\ln n)^{1/5})$.
\end{theorem}

The proof of Theorem~\ref{thm:main} is based on analysing the following
iterative process, which simulates the early stages of the $K_4$-free process.
This iterative process will be referred to throughout the paper simply as the
process.  Let $\epsilon_1, \epsilon_2$ and $\epsilon_3$ be positive constants
such that $\epsilon_1$ is sufficiently small with respect to $\epsilon_2 \deq
10^4 \epsilon_3^3$, and $\epsilon_3$ is sufficiently small.  Let $I \deq
\floor{ n^{\epsilon_1 + \epsilon_1^2} }$.  Let $\M_0$ and $\Trav_0$ be two
empty graphs.  Let $0 \le i < I$ be an integer and suppose we have already
defined $\M_i$ and $\Trav_i$. We define $\M_{i+1}$ and $\Trav_{i+1}$ as
follows.  Let $\NotTrav_i \deq \binom{[n]}{2} \setminus \Trav_i$.  Let
$\BIGBite_{i+1}$ be a random set of edges constructed by taking every edge in
$\NotTrav_i$ independently with probability $n^{\epsilon_3 - 2/5}$.  Let
$\BigBite_{i+1}$ be a random set of edges constructed by taking every edge in
$\BIGBite_{i+1}$ independently with probability $n^{\epsilon_2 - \epsilon_3}$.
Let $\Bite_{i+1}$ be a random set of edges constructed by taking every edge in
$\BigBite_{i+1}$ independently with probability $n^{-\epsilon_1 - \epsilon_2} /
(1 - i n^{-\epsilon_1 - \epsilon_2})$.\footnote{Equivalently, we could have
defined $\Bite_{i+1}$ to be a random set of edges constructed by taking every
edge in $\NotTrav_i$ independently with probability $n^{-\epsilon_1 - 2/5} / (1
- in^{-\epsilon_1 - \epsilon_2})$.  The intermediate sets $\BigBite_{i+1}$ and
$\BIGBite_{i+1}$ are introduced for technical reasons.} Assign each edge in
$\Bite_{i+1}$ a uniformly random birthtime in the unit interval and order the
edges in $\Bite_{i+1}$ by increasing birthtimes; traverse these edges according
to that order and add each traversed edge to $\M_i$, unless the addition of the
edge creates a copy of $K_4$. Let $\M_{i+1}$ be the graph thus constructed.
Finally, let $\Trav_{i+1} \deq \Trav_i \cup \Bite_{i+1}$ be the graph which is
the set of edges that were already traversed.

Let $s \deq C n^{3/5} (\ln n)^{1/5}$ be an integer, where $C = C(\epsilon_1)$
is a sufficiently large constant. Observe that in order to prove
Theorem~\ref{thm:main}, it is enough to prove the following theorem, which also
easily implies Theorem~\ref{thm:ramsey} (as the existence of a $K_4$-free
$n$-vertex graph in which every set of $s$ vertices spans a triangle implies
$f_{3,4}(n) < s$).
\begin{theorem} \label{thm:main2}
A.a.s, every set of $s$ vertices of $\M_I$ spans a triangle.
\end{theorem}

The rest of the paper is devoted for the proof of Theorem~\ref{thm:main2}.  The
argument that underlies the proof of Theorem~\ref{thm:main2} is an extension of
the branching process argument of Spencer~\cite{Spencer0a}, which was used to
give a limited, though non-trivial, analysis of the triangle-free process.  We
remark that our argument is different than the one used by
Bohman~\cite{MR2522430} to analyze the $K_4$-free process, an argument which is
based on the differential equations method.  Still, we note that there are
similarities between the two arguments and so, wherever possible, we will reuse
some of Bohman's results, instead of reproving them.

The paper is organized as follows.  In Section~\ref{sec:3} we state several
probabilistic tools that we use throughout the paper.  In Section~\ref{sec:2}
we give some basic definitions and state our main lemma.  In
Section~\ref{sec:5} we study a certain branching process and a certain
event~--~the event of survival~--~and in Section~\ref{sec:6} we relate that
event to the process.  In Section~\ref{sec:4} we prove several supporting
lemmas that will be used in the proof of our main lemma. In Section~\ref{sec:7}
we use the results of the preceding sections to prove our main lemma.  In
Section~\ref{sec:8} we use our main lemma to prove Theorem~\ref{thm:main2}.

\section{Probabilistic tools} \label{sec:3}
%
Here we state several probabilistic tools that we use throughout the paper.  We
start with two deviation inequalities for random variables that count small
subgraphs in the binomial random graph $G(n,p)$.  (As usual, the binomial
random graph $G(n,p)$ is the graph obtained by taking every edge in
$\binom{[n]}{2}$ independently with probability $p$.) Our setting is as
follows.  For some index set $L$, let $\{G_l : l \in L\}$ be a family
(potentially a multiset) of subgraphs of $\binom{[n]}{2}$, each of size $K =
O(1)$.  Let $W$ count the number of indices $l \in L$ such that $G_l \subseteq
G(n,p)$. In other words, let $W \deq \sum_{l \in L} \ONE[G_l \subseteq
G(n,p)]$, where $\ONE[\E]$ is the indicator function of the event $\E$.

The first deviation inequality follows directly from a more general result of
Vu~\cite[Corollary~4.4]{Vu02}.  For $G \subseteq \binom{[n]}{2}$, let $L_G$ be
the set of all $l \in L$ such that $G \subseteq G_l$. Let $W_G \deq \sum_{l \in
L_G} \ONE[G_l \setminus G \subseteq G(n,p)]$.  For an integer $0 \le k \le K$,
let $\expec_{k}(W) \deq \max_{G:|G| \ge k} \expec(W_G)$.  
\begin{theorem} \label{thm:vu}
Let $\EE_0 > \EE_1 > \ldots > \EE_K$ and $\lambda$ be positive numbers such
that $\EE_k \ge \expec_k(W)$ for all $0 \le k \le K$, and $\EE_k / \EE_{k+1}
\ge \lambda + 8 k \ln n$ for all $0 \le k \le K-1$. Then for some positive
constants $c_1$ and $c_2$ that depend only on $K$,
\begin{eqnarray*}
\prob(|W - \expec(W)| \ge c_1 \sqrt{\lambda \EE_0 \EE_1}) \le c_2 \exp(-\lambda
/ 4).
\end{eqnarray*}
\end{theorem}

The second deviation inequality follows directly from the more general Janson's
inequality~\cite{MR1138428} (see also~\cite[Theorem~2.14]{MR1782847}).  Let
$\Delta \deq \sum_{\{l, l'\}} \expec(\ONE[G_l, G_{l'} \subseteq G(n,p)])$, where
the sum ranges over all sets $\{l, l'\} \subseteq L$ such that $l \ne l'$ and
$G_l \cap G_{l'} \ne \emptyset$. 
\begin{theorem} \label{eq:janson}
For some absolute positive constant $c_3$, for all $0 \le \lambda \le
\expec(W)$,
\begin{eqnarray*}
\prob(W \le \expec(W) - \lambda) \le \exp\bigg(- \frac{c_3 \lambda^2}{\expec(W)
+ \Delta} \bigg).
\end{eqnarray*}
\end{theorem}

Another result that we use applies to the same setting as above, and follows
directly from a result of Janson and Ruci{\'n}ski~\cites{MR2096818, MR1900611}. 
\begin{theorem} \label{thm:deletion}
Assume that $\{G_l : l \in L\}$ is a set (that is, not a multiset).  For every
pair of positive real numbers $r$ and $\lambda$, with probability at least $1 -
\exp\big(- \frac{r \lambda }{K(\expec(W) + \lambda)}\big)$, there is a set $E_0
\subseteq G(n,p)$ of size at most $r$, such that $G(n,p) \setminus E_0$
contains fewer than $\expec(W) + \lambda$ members of $\{G_l : l \in L\}$.
\end{theorem}

We end this section by stating McDiarmid's inequality~\cite{MR1036755}.  Let
$\alpha_1, \alpha_2, \ldots, \alpha_m$ be independent random variables with
$\alpha_i$ taking values in a set $A_i$. Let $\varphi : \prod_{i=1}^m A_i \to
\RR$ satisfy the following Lipschitz condition: if two vectors $\alpha, \alpha'
\in \prod_{i=1}^m A_i$ differ only in the $i$th coordinate, then
$|\varphi(\alpha) - \varphi(\alpha')| \le a_i$. Redefine $W \deq
\varphi(\alpha_1, \alpha_2,\ldots, \alpha_m)$.  McDiarmid's inequality states
that for any $\lambda \ge 0$,
\begin{eqnarray*} \label{eq:mcdiarmid} 
\prob(|W - \expec(W)| \ge \lambda) \le 2 \exp\bigg(
-\frac{2\lambda^2}{\sum_{i=1}^m a_i^2} \bigg).  
\end{eqnarray*}

\section{Main lemma} \label{sec:2}
%
The purpose of this section is to state our main lemma.  We start with some
basic definitions.  Let $\Phi(x)$ be a function over the Reals, whose
derivative is $\phi(x) \deq \exp(-0.5 \Phi(x)^5)$, and which satisfies $\Phi(0)
= 0$. For integers $i$, $j$ and $t$, let
\begin{eqnarray*}
x_{i,j}  &\deq&  \binom{n}{2} \binom{5}{j}
\bigg(\frac{\Phi(in^{-\epsilon_1})}{n^{2/5}}\bigg)^{5-j}
\phi(in^{-\epsilon_1})^j  \text{\quad and \quad } 
\\ 
y_{i,j,t}  &\deq& t \binom{3}{j} \bigg( \frac{\Phi(in^{-\epsilon_1})}{n^{2/5}}
\bigg)^{3-j} \phi(in^{-\epsilon_1})^j.
\end{eqnarray*}
In addition, let 
\begin{eqnarray*}
&& z_{i,j} \deq n^{(\epsilon_2 - 2/5)j} x_{i,j}, \quad
\gamma_i \deq 2 \prod_{1 \le j \le 5} \big(1 - 2 n^{-\epsilon_1j -
\epsilon_2j})^{-6000 z_{i,j}} - 2 \text{\quad and } \\
&& \Gamma_i \deq \left\{ \begin{array}{ll} n^{-\epsilon_1} & \text{if } i = 0,
\\ \Gamma_{i-1}  (1 + \gamma_{i-1}) & \text{if } i \ge 1.  \end{array} \right.
\end{eqnarray*}
We remark that for every $0 \le i \le I$, $\Gamma_i \to 0$ as $n \to \infty$.
This is implied by the next lemma (whose proof is given below), which also
gives other useful facts that will be used throughout the paper.
\begin{lemma} \label{fact:f1}  For an integer $0 \le i \le I$,
\begin{itemize}
\item[(i)] $in^{-\epsilon_1} \to \infty \implies \phi(in^{-\epsilon_1}) =
\Theta ( 1 / (i n^{-\epsilon_1} (\ln in^{-\epsilon_1} )^{4/5}))$ and
$\Phi(in^{-\epsilon_1}) = \Theta( (\ln i n^{-\epsilon_1})^{1/5} )$;
\item[(ii)] $n^{-\Theta(\epsilon_1)} \le \phi(in^{-\epsilon_1}) \le 1$ and $i
\ge 1 \implies 0.9n^{-\epsilon_1} \le \Phi(in^{-\epsilon_1}) = O((\ln
n)^{1/5})$;
\item[(iii)] $\gamma_i = n^{-\Theta(\epsilon_1)}$ and $\Gamma_i =
n^{-\Theta(\epsilon_1)}$.
\end{itemize}
\end{lemma}

For an integer $0 \le i < I$, let $O_i$ be the set of edges $f \in \NotTrav_i$
such that $\M_i \cup \{f\}$ is $K_4$-free.  Furthermore, for an integer $0 \le
j \le 5$ and an edge $f \in \NotTrav_i$, let $X_{i,j}(f)$ be the set of all
graphs $G \subseteq \M_i \cup \NotTrav_i$ such that $|G| = 5$, $G \cup \{f\}$
is isomorphic to $K_4$, $|\M_i \cap G| = 5 - j$, $|\NotTrav_i \cap G| = j$, and
$\M_i \cup \{g\}$ is $K_4$-free for all $g \in G$ (or equivalently, $g \in O_i$
for all $g \in \NotTrav_i \cap G$).  Lastly, let $O'_i$ be the set of all $f
\in O_i$ such that $f \in \BIGBite_{i+1}$, and let $X'_{i,j}(f)$ be the set of
all $G \in X_{i,j}(f)$ such that $G \subseteq \M_i \cup \BIGBite_{i+1}$.

For integers $0 \le i < I$ and $1 \le j \le 3$, and for a set of $T$ of
triangles in $\binom{[n]}{2}$, let $Y_{i,j}(T)$ be the set of all triangles $G
\in T$ such that $|\M_i \cap G| = 3-j$, $|\NotTrav_i \cap G| = j$, and $\M_i
\cup G$ is $K_4$-free.  
Furthermore, for an integer $1 \le k < j$, let $Y_{i,j,k}(T)$ be the set of all
triples $(G_1, G_2, G_3)$ such that for some $G \in Y_{i,j}(T)$, $G_1 = \M_i
\cap G$, $G_2 \subset \NotTrav_i \cap G$ with $|G_2| = k$, and $G_3 = G
\setminus (G_1 \cup G_2)$.  (Note that $|Y_{i,j,k}(T)| = \binom{j}{k}
|Y_{i,j}(T)|$.) Lastly, let $Y'_{i,j,k}(T)$ be the set of all triples $(G_1,
G_2, G_3) \in Y_{i,j,k}(T)$ such that $G_2 \subseteq \BIGBite_{i+1}$.

For an integer $0 \le i < I$, for a set $R \subseteq [n]$, and for a set $T$ of
triangles in $\binom{R}{2}$, let $Z_i(R, T)$ be the set of all triangles $G \in
T$ such that $|\M_i \cap G| = 2$, $|\NotTrav_i \cap G| = 1$, and letting $g$
denote the edge in $\NotTrav_i \cap G$, there exists $G_0 \in X_{i,0}(g)$ such
that $G_0$ shares at least three vertices with $R$.

For a set $S \subseteq [n]$ with $s - o(s) \le |S| \le s$, let $\Pairs(S)$ be
the set of all pairs $(R, T)$ such that $R \subseteq S$ has size $s - o(s) \le
|R| \le |S|$, and for some partition of $R$ to three sets of size $\Omega(s)$
each, $T$ is the set of all triangles in $\binom{R}{2}$ that have each exactly
one vertex in each part of the partition.

Next, we define a few events. These events will be used to track the random
variables defined above, as they evolve throughout the process, as well as to
track some other properties of the process.  In what follows, and throughout
the paper, an expression that contains a symbol $\pm$ is a shorthand for the
interval $[\eta_-, \eta_+]$, where $\eta_-$ is obtained by replacing in the
expression every $\pm$ with $-$, and $\eta_+$ is obtained by replacing in the
expression every $\pm$ with $+$.  
\begin{itemize}
\item Let $\A_i$ be the event that the following properties hold:
\begin{itemize}
\item[(A1)] $|\M_i| \in 0.5 n^{8/5} \Phi(i n^{-\epsilon_1}) (1 \pm 100
\Gamma_i)$; 
\item[(A2)] $|O_i| \in 0.5 n^2 \phi(in^{-\epsilon_1}) (1 \pm 100 \Gamma_i)$; 
\item[(A3)] $|X_{i,j}(f)| \in x_{i,j} (1 \pm 1000 \Gamma_i)$ for all $1 \le j
\le 5$ and all $f \in \NotTrav_i$.
\end{itemize}
\item Let $\B_i$ be the event that for every set $S \subseteq [n]$ of $s$
vertices, there is a set $S_i \subseteq S$, such that the following properties
hold:
\begin{itemize}
\item[(B1)] $S_i$ has size at least $s(1 - i n^{-0.01})$;
\item[(B2)] $\Trav_i \cap \binom{S_i}{2}$ has maximum degree at most
$n^{1.1/5}$;
\item[(B3)] $|Y_{i,j}(T)| \ge y_{i,j,t} (1 - 100 \Gamma_i) - 0.5 j (3-j) (2-j)
|Z_i(R, T)|$ for all $1 \le j \le 3$, and for every pair $(R, T) \in
\Pairs(S_i)$ with $|T| = t$.  (Note that the coefficient of $|Z_i(R, T)|$ is
equal to $0$ if $2 \le j \le 3$ and is equal to $1$ if $j=1$.)
\end{itemize}
\item Let $\C_i$ be the event that the following properties hold:
\begin{itemize}
\item[(C1)] the number of edges in $\Trav_i \cup \BIGBite_{i+1}$ is at most
$n^{8/5 + 10 \epsilon_3}$;
\item[(C2)] for every set $S \subseteq [n]$ of $s$ vertices, there are at most
$n^{4/5 + 10 \epsilon_3}$ edges in $(\Trav_i \cup \BIGBite_{i+1}) \cap
\binom{S}{2}$;
\item[(C3)] for every two vertices $v_1, v_2 \in [n]$, there are at most
$n^{1/5 + 10\epsilon_3}$ other vertices in $[n]$ that are adjacent in $\Trav_i
\cup \BIGBite_{i+1}$ simultaneously to $v_1$ and $v_2$; moreover, for every
three vertices $v_1, v_2, v_3 \in [n]$, there are at most $(\ln n)^{O(1)}$
other vertices in $[n]$ that are adjacent in $\Trav_i \cup \BIGBite_{i+1}$
simultaneously to $v_1, v_2$ and $v_3$;
\item[(C4)] for every edge $f$, there are at most $(\ln n)^{O(1)}$ pairs $(G,
v)$ such that $G \in X_{0,5}(f)$, $G \subseteq \Trav_i \cup \BIGBite_{i+1}$,
and $v$ is a vertex outside of the vertex set (of size four) of $G$ which is
adjacent in $\Trav_i \cup \BIGBite_{i+1}$ to at least three vertices of $G$;
\item[(C5)] for every edge $f \in \Trav_i \cup \BIGBite_{i+1}$, there are at
most $n^{2/5 + 10\epsilon_3}$ copies of $K_4^-$ (which is a $K_4$ without an
edge) in $\Trav_i \cup \BIGBite_{i+1}$ which contain $f$;
\item[(C6)] for every set $S \subseteq [n]$ of $s$ vertices, there is a set of
at most $n^{3/5 + 10 \epsilon_3}$ edges, the removal of which from $\Trav_i
\cup \BIGBite_{i+1}$ leaves at most $n^{4/5 + 10 \epsilon_3}$ $4$-cycles in
$(\Trav_i \cup \BIGBite_{i+1}) \cap \binom{S}{2}$;
\item[(C7)] for every set $R \subseteq [n]$ of $r$ vertices, where $s - o(s)
\le r \le s$, the following holds for every $\M \subseteq \Trav_i \cup
\BIGBite_{i+1}$, assuming the maximum degree in $\M \cap \binom{R}{2}$ is at
most $n^{1.1/5}$.  First, there are at most $O(n^{4.2/5})$ edges $g \in
\binom{R}{2}$ for which there exists a graph $G \in X_{0,5}(g)$, with $G
\subseteq \M$, which shares all four vertices with $R$.  Second, there is a set
$R_0 \subseteq [n] \setminus R$ of at most $n^{0.99/5}$ vertices, such that
there are at most $O(n^{4.2/5})$ edges $g \in \binom{R}{2}$ for which there
exists a graph $G \in X_{0,5}(g)$, with $G \subseteq \M$, which shares exactly
three vertices with $R$ and one vertex with $[n] \setminus (R \cup R_0)$; 
\item[(C8)] for every set $R \subseteq [n]$ of $r$ vertices, where $s - o(s)
\le r \le s$, and for every set $E$ of $O(n^{1/2})$ edges in $\binom{[n]}{2}
\setminus \binom{R}{2}$, the following holds, assuming the maximum degree in
$(\Trav_i \cup \BIGBite_{i+1}) \cap \binom{R}{2}$ is at most $n^{1.1/5}$. There
are at most $O(n^{4.2/5})$ edges $g \in \binom{R}{2}$ for which there exists a
path of length two in $(\Trav_i \cup \BIGBite_{i+1}) \cap \binom{R}{2}$ that
completes $g$ to a triangle, and a graph $G \in X_{0,5}(g)$, with $G \subseteq
(\Trav_i \cup \BIGBite_{i+1}) \setminus \binom{R}{2}$ and $G \cap E \ne
\emptyset$.
\end{itemize}
\item Let $\D_i$ be the event that the following properties hold: 
\begin{itemize}
\item[(D1)] $|O'_i| \in 0.5 n^{8/5 + \epsilon_3} \phi(i n^{-\epsilon_1}) (1 \pm
(100 \Gamma_i + \Gamma_i\gamma_i))$;
\item[(D2)] $|X'_{i,j}(f)| \in n^{(\epsilon_3 - 2/5)j} x_{i,j} (1 \pm 2000
\Gamma_i)$ for all $1 \le j \le 5$ and all $f \in \NotTrav_i$. In particular,
$|X'_{i,j}(f)| \le n^{\epsilon_3 j + o(1)} \le n^{10 \epsilon_3}$ for all $1
\le j \le 5$ and all $f \in \NotTrav_i$;
\item[(D3)] for every set $S \subseteq [n]$ of $s$ vertices, assuming $\B_i$ holds,
letting $S_i \subseteq S$ be the set that is guaranteed to exist by $\B_i$,
\begin{eqnarray*}
|Y'_{i,j,k}(T)| \ge n^{(\epsilon_3 - 2/5)k} \tbinom{j}{k} y_{i,j,t} (1 -
100\Gamma_i - \Gamma_i \gamma_i )
\end{eqnarray*}
for all $1 \le k < j \le 3$, and for every pair $(R, T) \in \Pairs(S_i)$ with
$|T| = t$.
\end{itemize}
\end{itemize}

Finally, we state our main lemma.
\begin{lemma} \label{lemma0} 
For $0 \le i < I$, $\prob(\A_i \wedge \B_i \wedge \C_i \wedge \D_i) \ge 1 - i
n^{-0.1} - n^{-\omega(1)}$.
\end{lemma}

\subsection{Proof of Lemma~\ref{fact:f1}}
%
A standard analysis of the separable differential equation $\phi(x) = \exp (
-0.5 \Phi(x)^5 )$ with the initial condition $\Phi(0) = 0$ gives the estimates
in the first two items. It also shows that $\Phi(in^{-\epsilon_1})$ is
monotonically increasing with $i$. We prove the validity of the third item.

From the definition of $z_{i,j}$ it follows that
\begin{eqnarray*}
z_{i,j} = \Theta \big( n^{\epsilon_2 j} \cdot \Phi(in^{-\epsilon_1})^{5-j}
\cdot \phi(in^{-\epsilon_1})^j \big).
\end{eqnarray*}
Plugging this into the definition of $\gamma_i$, we get that for all $0 \le i
\le I$,
\begin{eqnarray*}
\gamma_i = 2 \exp\bigg( \Theta(1) \cdot \sum_{1 \le j \le 5} n^{-\epsilon_1j }
\cdot \Phi(in^{-\epsilon_1})^{5-j} \cdot \phi(in^{-\epsilon_1})^j \bigg) - 2.
\end{eqnarray*}
This, together with the second item, implies that for all $0 \le i \le I$,
\begin{eqnarray} \label{eq:gamma1}
n^{-\Theta(\epsilon_1)} \le \gamma_i = O( n^{-\epsilon_1} + n^{-\epsilon_1}
\cdot \Phi(in^{-\epsilon_1})^4 \cdot \phi(in^{-\epsilon_1}) ),
\end{eqnarray}
and if in addition $in^{-\epsilon_1} \to \infty$, then
\begin{eqnarray} \label{eq:gamma2}
\gamma_i = O( n^{-\epsilon_1} \cdot \Phi(in^{-\epsilon_1})^4 \cdot
\phi(in^{-\epsilon_1}) ).
\end{eqnarray}
In particular, from~(\ref{eq:gamma1}) and the second item, we get that
$\gamma_i = n^{-\Theta(\epsilon_1)}$ for all $0 \le i \le I$.

To show that $\Gamma_i = n^{-\Theta(\epsilon_1)}$ for all $0 \le i \le I$, it
suffices to show that $\Gamma_I \le n^{-\Theta(\epsilon_1)}$.  For $0 \le i \le
I$, by~(\ref{eq:gamma1}) and by the second item, 
\begin{eqnarray*}
\gamma_i = O(n^{-\epsilon_1} + n^{-\epsilon_1} \cdot \Phi(in^{-\epsilon_1})^4).
\end{eqnarray*}
Therefore, letting $i_0 \deq \floor{n^{\epsilon_1} \ln\ln n}$, for $0 \le i \le
i_0$, by the first item and the monotonicity of $\Phi(in^{-\epsilon_1})$,
$\gamma_i = O(n^{-\epsilon_1} \ln\ln n)$.  Thus,
\begin{eqnarray*}
\Gamma_{i_0} = n^{-\epsilon_1} \cdot \prod_{0 \le i < i_0} (1 + \gamma_i) =
n^{-\epsilon_1 + o(1)}.
\end{eqnarray*}
By~(\ref{eq:gamma2}) and the first item, if $i_0 \le i \le I$, then
\begin{eqnarray*}
\gamma_i = O\bigg(\frac{1}{i}\bigg).
\end{eqnarray*}
Therefore, recalling that $I = \floor{n^{\epsilon_1 + \epsilon_1^2}}$,
\begin{eqnarray*}
\Gamma_I  =  \Gamma_{i_0} \cdot \prod_{i_0 \le i  <  I} (1 + \gamma_i)  
 =  n^{-\epsilon_1 + o(1)} \cdot \exp\bigg( \sum_{i_0 \le i  <  I}
O\bigg(\frac{1}{i}\bigg) \bigg) 
= n^{-\epsilon_1 + O(\epsilon_1^2)}.
\end{eqnarray*}

\section{Survival} \label{sec:5}
%
The purpose of this section is to define and analyze a certain event, an event
which in the next section will be related to the process, and which later on
will be used in the analysis of the process.
Fix for the rest of the section an integer $0 \le i < I$ and a sufficiently
large integer $c$ which we will assume to be constant, independent of $n$.  Let
$\TT_1$ be a rooted, finite tree with the following three properties: first,
each leaf in the tree is at distance $2c$ from the root; second, every non-leaf
node at even distance from the root has $5$ sets of children, where the $j$th
set has size in $z_{i,j} (1 \pm 3000 \Gamma_i)$ and consists of sets of size
$j$; third, every node at odd distance from the root which is a set of size $j$
has exactly $j$ children.
(We remark that we don't use the fact that a node at odd distance from the root
of $\TT_1$ is a set of size $j$, except to indicate that such a node has $j$
children.)
Assign each node $\nu$ at even distance from the root of $\TT_1$ a uniformly
random birthtime $\beta(\nu)$ in the unit interval. 
Let $\nu_0$ be a node at even distance from the root of $\TT_1$.  Define the
event that $\nu_0$ survives as follows.  If $\nu_0$ is a leaf then $\nu_0$
survives by definition;  otherwise, $\nu_0$ survives if and only if for every
child $\nu_1$ of $\nu_0$, the following holds: if for every child $\nu_2$ of
$\nu_1$ we have $\beta(\nu_2) < \beta(\nu_0)$ then $\nu_1$ has a child that
does not survive. 

Let $\nu_{\TT_1}$ denote the root of $\TT_1$.  Let $p_{\TT_1}(x)$ be the
probability that $\nu_{\TT_1}$ survives, under the assumption that its
birthtime is equal to $x n^{-\epsilon_2} / (1 - i n^{-\epsilon_1 -
\epsilon_2})$; in other words, we define $\beta$ as above, only that now we
further set the birthtime of $\nu_{\TT_1}$ to be $x n^{-\epsilon_2} / (1 - i
n^{-\epsilon_1 - \epsilon_2})$, and ask for the probability $p_{\TT_1}(x)$ that
$\nu_{\TT_1}$ survives given that setup.  Let $P_{\TT_1}(x) \deq x
\prob(\text{$\nu_{\TT_1}$ survives} \given \beta(\nu_{\TT_1}) \le x
n^{-\epsilon_2} / (1 - i n^{-\epsilon_1 - \epsilon_2}))$ if $x > 0$ and
$P_{\TT_1}(x) \deq 0$ if $x=0$.  The main result of this section follows.
\begin{lemma} \label{lemma:survive}
$p_{\TT_1}(n^{-\epsilon_1}) \in
\frac{\phi((i+1)n^{-\epsilon_1})}{\phi(in^{-\epsilon_1})} (1 \pm 20 \Gamma_i
\gamma_i)$; $P_{\TT_1}(n^{-\epsilon_1}) \in \frac{\Phi((i+1)n^{-\epsilon_1}) -
\Phi(in^{-\epsilon_1})}{\phi(in^{-\epsilon_1})} (1 \pm 20 \Gamma_i \gamma_i)$.
\end{lemma}

\subsection{A related event}
Let ${\TT_2}$ be a rooted tree with the following three properties: first, each
leaf in the tree is at even distance from the root; second, every node at even
distance less than $2c$ from the root has $5$ sets of children, where the $j$th
set has size in $z_{i,j} (1 \pm 3000 \Gamma_i)$ and consists of sets of size
$j$; third, every node at odd distance less than $2c$ from the root which is a
set of size $j$ has exactly $j$ children.
(Note that we make no assumptions as for the number of children of nodes at
distance at least $2c$ from the root of ${\TT_2}$. In particular, it is
possible for ${\TT_2}$ to have an infinite path.) Extend $\beta$ by assigning
each node $\nu$ at even distance from the root of ${\TT_2}$ a uniformly random
birthtime $\beta(\nu)$ in the unit interval.  Define  the event that a node at
even distance from the root of ${\TT_2}$ survives exactly as it was defined for
such a node in ${\TT_1}$. Let $\nu_{\TT_2}$ denote the root of $\TT_2$.  Let
$p_{\TT_2}(x)$ be the probability that $\nu_{\TT_2}$ survives, under the
assumption that its birthtime is equal to $x n^{-\epsilon_2} / (1 - i
n^{-\epsilon_1 - \epsilon_2})$.  Let $P_{\TT_2}(x) \deq x
\prob(\text{$\nu_{\TT_2}$ survives} \given \beta(\nu_{\TT_2}) \le x
n^{-\epsilon_2} / (1 - i n^{-\epsilon_1 - \epsilon_2}))$ if $x > 0$ and
$P_{\TT_2}(x) \deq 0$ if $x=0$.  The purpose of this subsection is to prove the
following lemma, which relates the event that the root of $\TT_1$ survives to
the event that the root of $\TT_2$ survives.
\begin{lemma} \label{lemma:survive1}
$p_{\TT_1}(n^{-\epsilon_1}) \in p_{\TT_2}(n^{-\epsilon_1}) (1 \pm 5
\Gamma_i\gamma_i)$ and
$P_{\TT_1}(n^{-\epsilon_1}) \in P_{\TT_2}(n^{-\epsilon_1}) (1 \pm 5
\Gamma_i\gamma_i)$.
\end{lemma}
Let $\TT_3$ be obtained by removing from $\TT_2$ every subtree that is rooted
at a node at distance larger than $2c$ from $\nu_{\TT_2}$. Observe that $\TT_3$
satisfies the exact same three properties that $\TT_1$ satisfies. Define the
event that a node at even distance from the root of $\TT_3$ survives exactly as
it was defined for such a node in $\TT_1$. (Note that a node in $\TT_3$ is also
a node in $\TT_2$, but the event that such a node survives with $\TT_2$ being
the underlying tree is not necessarily the same as the event that such a node
survives with $\TT_3$ being the underlying tree. Below, when stating that a
node in $\TT_3$ survives, the exact tree that underlies this event should be
understood from the context. For example, when stating that $\nu_{\TT_2}$
survives we refer to the event that the root of $\TT_2$ survives and not to the
event that the root of $\TT_3$ survives.) Let $\nu_{\TT_3}$ denote the root of
$\TT_3$.  Let $p_{\TT_3}(x)$ be the probability that $\nu_{\TT_3}$ survives
under the assumption that its birthtime is equal to $x n^{-\epsilon_2} / (1 - i
n^{-\epsilon_1 - \epsilon_2})$. Let $P_{\TT_3}(x) \deq x \prob(
\text{$\nu_{\TT_3}$ survives} \given \beta( \nu_{\TT_3} ) \le x n^{-\epsilon_2}
/ (1 - i n^{-\epsilon_1 - \epsilon_2}))$ if $x > 0$ and $P_{\TT_3}(x) \deq 0$
if $x=0$.  The next two lemmas imply Lemma~\ref{lemma:survive1}.
\begin{lemma} 
For $0 \le x \le n^{-\epsilon_1}$, $p_{\TT_3}(x) \in p_{\TT_2}(x)(1 \pm
\Gamma_i\gamma_i)$.
\end{lemma}
\begin{proof}
We prove the lemma under the assumption that $c$ is odd. The proof for the case
where $c$ is even is similar and will be omitted.

Fix $0 \le x \le n^{-\epsilon_1}$ and assume that $\beta(\nu_{\TT_2}) =
\beta(\nu_{\TT_3}) = x n^{-\epsilon_2} / (1 - i n^{-\epsilon_1 - \epsilon_2}) <
2 x n^{-\epsilon_2}$.  Since $c$ is odd, the event that $\nu_{\TT_3}$ survives
implies the event that $\nu_{\TT_2}$ survives. Hence $p_{\TT_3}(x) \le
p_{\TT_2}(x)$.  Below we show that $p_{\TT_3}(x) \ge p_{\TT_2}(x) -
n^{-\Theta(\epsilon_1c)}$.  Note that $p_{\TT_2}(x) = \Omega(1)$. (Indeed, a
sufficient condition for the event that $\nu_{\TT_2}$ survives is that for
every child $\nu_1$ of $\nu_{\TT_2}$, there is a child $\nu_2$ of $\nu_1$ with
$\beta(\nu_2) > \beta(\nu_{\TT_2})$.  Given the above assumption on
$\beta(\nu_{\TT_2})$ and the properties of $\TT_2$, this event occurs with
probability $\Omega(1)$.) Therefore, $p_{\TT_3}(x) \ge p_{\TT_2}(x) ( 1 -
n^{-\Theta(\epsilon_1c)})$. Since $c$ is sufficiently large, it follows from
Lemma~\ref{fact:f1} that $\Gamma_i \gamma_i \ge n^{-\Theta(\epsilon_1 c)}$.
This gives the lemma.

Say that a non-root node $\nu$ at even distance from the root of $\TT_3$ is
relevant, if the following two properties hold: first, the grandparent of $\nu$
has a larger birthtime than the birthtime of $\nu$ and the birthtimes of
$\nu$'s siblings (if there are any); second the grandparent of $\nu$ is either
relevant or the root.  Observe that if the root of $\TT_2$ survives then either
the root of $\TT_3$ survives, or else, there is a relevant leaf in $\TT_3$.
Thus, it remains to show that the expected number of relevant leaves in $\TT_3$
is at most $n^{-\Theta(\epsilon_1c)}$.

Say that a leaf $\nu$ in $\TT_3$ is an {$(a_1, a_2, a_3, a_4, a_5)$-type}, if
the path leading from the root to $\nu$ contains exactly $a_j$ nodes at odd
distance from the root which are sets of size $j$. Consider a path
$(\nu_{\TT_3}, \nu_1, \nu_2, \ldots, \nu_{2c})$ from the root to a leaf, where
the leaf $\nu_{2c}$ is an $(a_1, a_2, a_3, a_4, a_5)$-type. Given such a path,
let $N$ be the set of nodes which is the union of $\{\nu_{2b}: 1 \le b \le c\}$
together with $\{\nu : \text{$\nu$ is a sibling of some $\nu_{2b}$ for some $1
\le b \le c$}\}$.  Note that $|N| = \sum_{1 \le j \le 5} j a_j = \Theta(c)$.
Now, if $\nu_{2c}$ is relevant, then the birthtime of every node in $N$ is less
than $2 x n^{-\epsilon_2}$.  This event occurs with probability $(2 x
n^{-\epsilon_2})^{|N|}$.  The number of $(a_1, a_2, a_3, a_4, a_5)$-type leaves
in $\TT_3$ is at most $4^{5c} \prod_{1 \le j \le 5} z_{i,j}^{a_j}$.  Hence, the
expected number of relevant $(a_1, a_2, a_3, a_4, a_5)$-type leaves in $\TT_3$
is at most
\begin{eqnarray*} 
(2 x n^{-\epsilon_2})^{|N|} \cdot 4^{5c} \prod_{1 \le j \le 5} z_{i,j}^{a_j} = 
4^{5c} \cdot \prod_{1 \le j \le 5} (2 x n^{-\epsilon_2})^{j a_j} z_{i,j}^{a_j}
\le n^{-\Theta(\epsilon_1c)},
\end{eqnarray*}
where the inequality follows since $z_{i,j} \le n^{\epsilon_2j + o(1)}$ by
Lemma~\ref{fact:f1} and since $x \le n^{-\epsilon_1}$.  To complete the proof,
note that if a leaf is an $(a_1, a_2, a_3, a_4, a_5)$-type then the number of
choices we have for $\{a_j : 1 \le j \le 5\}$ is at most $(c+1)^5$. A union
bound argument now finishes the proof.
\end{proof}

\begin{lemma}
For $0 \le x \le n^{-\epsilon_1}$,
$p_{\TT_1}(x) \in p_{\TT_3}(x)(1 \pm \Gamma_i\gamma_i)$.
\end{lemma}
\begin{proof}
For a node $\nu$ in a tree $\TT_*$, let $p_{\TT_*,\nu}(x)$ be the probability
that $\nu$ survives under the assumption that $\beta(\nu) = x n^{-\epsilon_2} /
(1 - i n^{-\epsilon_1 - \epsilon_2})$ and furthermore, let $P_{\TT_*,\nu}(x)
\deq x \prob(\text{$\nu$ survives} \given \beta(\nu) \le x n^{-\epsilon_2} / (1
- i n^{-\epsilon_1 - \epsilon_2}))$ if $x > 0$ and $P_{\TT_*,\nu}(x) \deq 0$ if
$x=0$.  The following implies the lemma.
\begin{claim}
Let $0 \le x \le n^{-\epsilon_1}$. Let $0 \le b \le c$ be an integer.  If $\nu$
is a node at height $2b$ in $\TT_1$ and $\mu$ is a node at height $2b$ in
$\TT_3$, then
$p_{\TT_1,\nu}(x) \in p_{\TT_3,\mu}(x) (1 \pm \Gamma_i \gamma_i)$.
\end{claim}
The proof of the claim is by induction on $b$. For $b=0$, both $\nu$ and $\mu$
are leaves and so the claim holds since by definition $p_{\TT_1,\nu}(x) =
p_{\TT_3,\mu}(x) = 1$ for all $0 \le x \le n^{-\epsilon_1}$.  Let $1 \le b \le
c$ be an integer and assume the claim holds for $b-1$, for all $0 \le x \le
n^{-\epsilon_1}$.  Note that by the induction hypothesis, if $\nu'$ is a node
at height $2(b-1)$ in $\TT_1$ and $\mu'$ is a node at height $2(b-1)$ in
$\TT_3$ then for all $0 \le x \le n^{-\epsilon_1}$,
\begin{eqnarray*}
P_{\TT_1,\nu'}(x) \in P_{\TT_3,\mu'}(x) (1 \pm \Gamma_i \gamma_i).
\end{eqnarray*}

Fix $0 \le x \le n^{-\epsilon_1}$, a node $\nu$ at height $2b$ in $\TT_1$ and a
node $\mu$ at height $2b$ in $\TT_3$.  Recall that $\TT_1$ and $\TT_3$ satisfy
the same properties, and so it is enough to prove that $p_{\TT_1, \nu} \le
p_{\TT_3, \mu} ( 1 + \Gamma_i \gamma_i)$. 
Let $\Children(\cdot)$ denote the set of children of a given node in either
$\TT_1$ or $\TT_3$.  Let $N_j \deq \{ \Children(\nu_1) : \text{$\nu_1$ is a
child of $\nu$ which is a set of size $j$ } \}$, and likewise, let $L_j \deq \{
\Children(\mu_1) : \text{$\mu_1$ is a child of $\mu$ which is a set of size $j$
} \}$.
It is safe to assume that $z_{i,j}(1 - 3000 \Gamma_i) \le |N_j| \le |L_j|$ for
every $j$ (since otherwise we can remove some of the subtrees that are rooted
at some of the children of $\nu$ so that this assumption does hold; such an
alteration will only increase the probability that $\nu$ survives).  Let $l_j$
be an injective function that associates each set in $N_j$ with a unique set in
$L_j$.  For brevity, let $\zeta \deq 1 / (1 - i n^{-\epsilon_1 - \epsilon_2})$.
We have
\begin{eqnarray*}
p_{\TT_1,\nu}(x) & = & \prod_{1 \le j \le 5} \prod_{S \in N_j} \Big(1 - \zeta^j
n^{-\epsilon_2j} \prod_{\nu' \in S} P_{\TT_1,\nu'}(x) \Big) \\
&\le& \prod_{1 \le j \le 5} \prod_{S \in N_j} \Big(1 - \zeta^j n^{-\epsilon_2j}
\prod_{\mu' \in l_j(S)} \big( P_{\TT_3,\mu'}(x) (1 - \Gamma_i \gamma_i) \big)
\Big) \\ 
&\le& \prod_{1 \le j \le 5} \prod_{S \in N_j} \Big(1 - \zeta^j n^{-\epsilon_2j}
\prod_{\mu' \in l_j(S)} P_{\TT_3,\mu'}(x) \Big)^{1 - 2 j \Gamma_i \gamma_i} =
(*),
\end{eqnarray*}
where the first equality is by definition, the first inequality is by the
induction hypothesis, and the second inequality follows from known exponential
inequalities (i.e., the fact that for $a > 1$, $\exp(-1/(a-1)) \le 1-1/a \le
\exp(-1/a)$), together with the fact that $\zeta^j n^{-\epsilon_2j} =
o(\Gamma_i \gamma_i)$ (which follows from Lemma~\ref{fact:f1}) and the fact
that $P_{\TT_3, \mu'}(x) \le x \le n^{-\epsilon_1}$ for all $\mu'$.
This upper bound on $P_{\TT_3,\mu'}(x)$, together with the fact that $|N_j| \le
2 z_{i,j}$ and with Lemma~\ref{fact:f1}, gives
\begin{eqnarray*} 
\prod_{1 \le j \le 5} \prod_{S \in N_j} \Big(1 - \zeta^j n^{-\epsilon_2j}
\prod_{\mu' \in l_j(S)} P_{\TT_3,\mu'}(x) \Big)^{- 2 j \Gamma_i \gamma_i} &\le&
\prod_{1 \le j \le 5} \prod_{S \in N_j} \Big(1 - \zeta^j n^{-\epsilon_1j -
\epsilon_2j} \Big)^{- 2 j \Gamma_i \gamma_i} \\
&\le& \prod_{1 \le j \le 5} \Big(1 - \zeta^j n^{-\epsilon_1j - \epsilon_2j}
\Big)^{- 4 j \Gamma_i \gamma_i z_{i,j}} \\
&\le& 1 + o(\Gamma_i \gamma_i).  
\end{eqnarray*}
Moreover, letting $L_j' \deq L_j \setminus \{ l_j(S) : S \in N_j \}$, we have
\begin{eqnarray*}
\prod_{1 \le j \le 5} \prod_{S \in N_j} \Big(1 - \zeta^j n^{-\epsilon_2j}
\prod_{\mu' \in l_j(S)} P_{\TT_3,\mu'}(x) \Big) &=& 
p_{\TT_3, \mu}(x) \prod_{1 \le j \le 5} \prod_{S \in L_j'} \Big(1 - \zeta^j
n^{-\epsilon_2j} \prod_{\mu' \in S}  P_{\TT_3,\mu'}(x) \Big)^{-1} \\ 
&\le& p_{\TT_3, \mu}(x) \prod_{1 \le j \le 5} \Big(1 - 2 n^{-\epsilon_1j -
\epsilon_2j} \Big)^{ - 6000 \Gamma_i z_{i,j}} \\ 
&\le& p_{\TT_3, \mu}(x) (1 + 0.5 \Gamma_i \gamma_i),  
\end{eqnarray*}
where the equality follows from the definition of $p_{\TT_3,\mu}(x)$, the first
inequality follows since $\zeta^j \le 2$, since $P_{\TT_3, \mu'}(x) \le x \le
n^{-\epsilon_1}$ for all $\mu'$, and since $|L_j'| = |L_j| - |N_j| \le 6000
\Gamma_i z_{i,j}$, and the second inequality follows using the definition of
$\gamma_i$.  It follows that $(*) \le p_{\TT_3, \mu}(x) (1 + \Gamma_i
\gamma_i)$.
\end{proof}

\subsection{Proof of Lemma~\ref{lemma:survive}} \label{sec:4:2}
%
Let $\TT_4$ be an infinite tree with the following two properties: first, each
node at even distance from the root has $0.5 n^{5\epsilon_2}$ children; second,
each node at odd distance from the root has $5$ children.  (Note that we
implicitly assume that $0.5 n^{5\epsilon_2}$ is an integer. This is a safe
assumption since we can always choose $\epsilon_2$ so that this assumption
holds.)
Extend $\beta$ by assigning each node $\nu$ at even distance from the root of
$\TT_4$ a uniformly random birthtime $\beta(\nu)$ in the unit interval.  Define
the event that a node at even distance from the root of $\TT_4$ survives
exactly as it was defined for such a node in $\TT_1$.
Let $\nu_{\TT_4}$ denote the root of $\TT_4$.  Let $p_{\TT_4}(x)$ be the
probability that $\nu_{\TT_4}$ survives under the assumption that its birthtime
is equal to $x n^{-\epsilon_2}$, at the limit as $n \to \infty$.  It is not
hard to see that $p_{\TT_4}(x)$ is continuous and bounded.  Hence
$p_{\TT_4}(x)$ is integrable.  Let $P_{\TT_4}(x) \deq \int_0^x p_{\TT_4}(y)
dy$. Note that for $0 < x \le n^{\epsilon_2}$, $P_{\TT_4}(x) = x
\prob(\text{$\nu_{\TT_4}$ survives} \given \beta(\nu_{\TT_4}) \le x
n^{-\epsilon_2})$.

\begin{lemma} \label{lemma:survive2}
For $0 \le x \le n^{\epsilon_2}$, $p_{\TT_4}(x) = \phi(x)$ and $P_{\TT_4}(x) =
\Phi(x)$.
\end{lemma}
\begin{proof}
By definition, for every $0 \le x \le n^{\epsilon_2}$,
\begin{eqnarray} \label{eq:g0} 
p_{\TT_4}(x) = \lim_{n \to \infty} \Big(1 - n^{-5\epsilon_2} P_{\TT_4}(x)^5
\Big)^{0.5 n^{5\epsilon_2}}  = \exp\Big( - 0.5 P_{\TT_4}(x)^5 \Big).  
\end{eqnarray}
(Indeed, if $0 < x \le n^{\epsilon_2}$ then the first equality above holds by
definition of the event that the root of $\TT_4$ survives; if $x = 0$ then the
first equality above holds since $p_{\TT_4}(0) = 1$ and $P_{\TT_4}(0) = 0$.)

By the fundamental theorem of calculus, $p_{\TT_4}(x)$ is the derivative of
$P_{\TT_4}(x)$.  Hence, we view~(\ref{eq:g0}) as the           differential
equation that it is.  Since $P_{\TT_4}(0) = 0$, by definition the solution of
this differential equation is $p_{\TT_4}(x) = \phi(x)$ and $P_{\TT_4}(x) =
\Phi(x).$
\end{proof}

Let $\S_1$ (respectively $\S_2$) be the event that the root of $\TT_4$ survives
under the assumption that its birthtime is equal to $i n^{-\epsilon_1 -
\epsilon_2}$ (respectively $(i+1) n^{-\epsilon_1 - \epsilon_2}$).  Let $\S_3$
(respectively $\S_4$) be the event that the root of $\TT_4$ survives,
conditioned on the event that its birthtime is at most $i n^{-\epsilon_1 -
\epsilon_2}$ (respectively $(i+1) n^{-\epsilon_1 - \epsilon_2}$), unless $i=0$
in which case we let $\S_3$ be the empty event.  Let $\S_5$ be the event that
the root of $\TT_4$ survives, conditioned on the event that its birthtime is in
$[in^{-\epsilon_1 - \epsilon_2}, (i+1) n^{-\epsilon_1 - \epsilon_2}]$.
By Lemma~\ref{lemma:survive2} and since $\prob(\S_2) = \prob(\S_1) \prob(\S_2
\given \S_1)$ and $\prob(\S_4) = \frac{i}{i+1} \prob(\S_3)  +  \frac{1}{i+1}
\prob(\S_1) \prob(\S_5 \given \S_1)$, we have
\begin{eqnarray} \label{eq:g1}
\prob(\S_2 \given \S_1) =  \frac{\phi((i+1)n^{-\epsilon_1})}{
\phi(in^{-\epsilon_1})} \quad \text{ and } \quad
\prob(\S_5 \given \S_1) = \frac{\Phi((i+1)n^{-\epsilon_1}) -
\Phi(in^{-\epsilon_1})}{n^{-\epsilon_1} \phi(in^{-\epsilon_1})}.
\end{eqnarray}

Let us consider  the events $\S_1$, $\S_2$ and $\S_5$.  These events depend on
the random function $\beta$. More accurately, these events depend on the
birthtimes of the nodes that are at even distances from the root of $\TT_4$. 
For the purpose of giving a few observations, let us imagine in this paragraph
that we can access $\beta$ through two different oracles.  The revealing-oracle
reveals everything: given a node $\nu$ it returns its birthtime $\beta(\nu)$.
The hiding-oracle does not reveal everything: given a node $\nu$ it returns its
birthtime $\beta(\nu)$ only if its birthtime is at most $in^{-\epsilon_1 -
\epsilon_2}$; otherwise it returns ``hidden'' (in which case one only learns
that $\beta(\nu) > in^{-\epsilon_1 - \epsilon_2}$).
Observe that in order to determine the occurrence of $\S_1$, it suffices to
only consult the hiding-oracle. 
In contrast, in order to determine the occurrence of $\S_2$ and $\S_5$, it is
not sufficient in general to only consult the hiding-oracle, as these two
events may depend on the exact birthtimes of nodes whose birthtimes are larger
than $in^{-\epsilon_1 - \epsilon_2}$; it is, however, sufficient to first
consult the hiding-oracle, to verify using the information obtained from the
hiding-oracle that $\S_1$ occurs (this is a necessary condition for the
occurrence of both $\S_2$ and $\S_5$), and then consult the revealing-oracle
for the birthtimes of all nodes whose exact birthtimes were not revealed by the
hiding-oracle.
The point we'd like to make is that after consulting the hiding-oracle and
verifying that $\S_1$ occurs, there are some nodes in $\TT_4$ whose birthtimes
need not be queried via the revealing-oracle in order to determine the
occurrence of $\S_2$ and $\S_5$.  We describe these nodes now.
Let $\nu$ be a non-root node at even distance from the root of $\TT_4$.  In
order to determine whether or not the grandparent of $\nu$ survives, we are
interested in knowing (among other things) whether or not the following holds:
$\nu$ and its siblings all have birthtimes smaller than that of their
grandparent, and all survive. From this we get the following observations.
If $\beta(\nu) \le in^{-\epsilon_1 - \epsilon_2}$ and we know that $\nu$
survives given only the information provided by the hiding-oracle, then in
order to determine the occurrence of $\S_2$ and $\S_5$, we may ignore the
subtree rooted at $\nu$ upon querying the revealing-oracle.
Further, if $\beta(\nu) \le in^{-\epsilon_1 - \epsilon_2}$ and we know that
$\nu$ does not survive given only the information provided by the
hiding-oracle, then in order to determine the occurrence of $\S_2$ and $\S_5$,
we may ignore the subtree rooted at the parent of $\nu$ (and in particular,
ignore the subtrees rooted at $\nu$ and its siblings) upon querying the
revealing-oracle.
Lastly, if $\beta(\nu) > in^{-\epsilon_1 - \epsilon_2}$ and $\nu$ has a child
$\nu_1$, such that  given only the information provided by the hiding-oracle we
know that for every child $\nu_2$ of $\nu_1$ it holds that $\beta(\nu_2) \le
in^{-\epsilon_1 - \epsilon_2}$ and $\nu_2$ survives, then in order to determine
the occurrence of $\S_2$ and $\S_5$, we may ignore the subtree rooted at the
parent of $\nu$ upon querying the revealing-oracle.  These observations
motivate the next definition.

Let $\TT_5$ be a random rooted tree (depending on the random function $\beta$)
that is obtained from $\TT_4$ using the following procedure. For every non-root
node $\nu$ in $\TT_4$ at even distance from the root, do: if $\beta(\nu) \le i
n^{-\epsilon_1 - \epsilon_2}$ and $\nu$ survives then remove the subtree rooted
at $\nu$, and if $\nu$ doesn't survive then remove the subtree rooted at the
parent of $\nu$; if $\beta(\nu) > i n^{-\epsilon_1 - \epsilon_2}$, and there is
a child $\nu_1$ of $\nu$ such that for every child $\nu_2$ of $\nu_1$ it holds
that $\beta(\nu_2) \le i n^{-\epsilon_1 - \epsilon_2}$ and $\nu_2$ survives
(which is the same as saying that $\nu$ doesn't survive under the assumption
that its birthtime is exactly $in^{-\epsilon_1 - \epsilon_2}$), then remove the
subtree rooted at the parent of $\nu$.  This gives the random rooted tree
$\TT_5$.
Assign each node $\nu$ at even distance from the root of $\TT_5$ a uniformly
random birthtime $\beta'(\nu)$ in the unit interval. 
Define the event that a node at even distance from the root of $\TT_5$ survives
exactly as it was defined for such a node in $\TT_1$, only that in the current
definition we replace $\beta$ with $\beta'$.  
%
%
Let $\S_6$ be the event that the root of $\TT_5$ survives under the assumption
that its birthtime under $\beta'$ is equal to $n^{-\epsilon_1 - \epsilon_2} /
(1 - i n^{-\epsilon_1 - \epsilon_2})$. 
Let $\S_7$ be the event that the root of $\TT_5$ survives, conditioned on the
event that its birthtime under $\beta'$ is at most $n^{-\epsilon_1 -
\epsilon_2} / (1 - i n^{-\epsilon_1 - \epsilon_2})$.
Given the discussion in the previous paragraph, we observe that $\prob(\S_2) =
\prob(\S_1 \wedge \S_6)$ and that $\prob(\S_5) = \prob(\S_1 \wedge \S_7)$.
These two equalities, together with the fact that $\S_1$ is implied by both
$\S_2$ and $\S_5$, give 
\begin{equation} \label{eq:g2}
\begin{split}
\prob(\S_6 \given \S_1) = \frac{\prob(\S_2)}{\prob(\S_1)} = \frac{\prob(\S_1
\wedge \S_2)}{\prob(\S_1)} = \prob(\S_2 \given \S_1) \text{\quad and } \\
\prob(\S_7 \given \S_1) = \frac{\prob(\S_5)}{\prob(\S_1)} = \frac{\prob(\S_1
\wedge \S_5)}{\prob(\S_1)} = \prob(\S_5 \given \S_1).
\end{split}
\end{equation}
From~(\ref{eq:g1})~and~(\ref{eq:g2}) we get
\begin{eqnarray} \label{eq:g3}
\prob(\S_6 \given \S_1) =  \frac{\phi((i+1)n^{-\epsilon_1})}{
\phi(in^{-\epsilon_1})} \quad \text{ and } \quad
\prob(\S_7 \given \S_1) = \frac{\Phi((i+1)n^{-\epsilon_1}) -
\Phi(in^{-\epsilon_1})}{n^{-\epsilon_1} \phi(in^{-\epsilon_1})}.
\end{eqnarray}
In addition, by~(\ref{eq:g2}), Lemma~\ref{lemma:survive2} (which implies
$\prob(\S_2) = \phi((i+1) n^{-\epsilon_1}))$, Lemma~\ref{fact:f1} (which
implies $\phi((i+1) n^{-\epsilon_1}) \ge n^{-\Theta(\epsilon_1)}$), and the
fact that $\S_2$ implies $\S_5$,
\begin{equation} \label{eq:g4}
\begin{split}
\prob(\S_6 \given \S_1) \ge \prob(\S_2) \ge n^{-\Theta(\epsilon_1)} \text{\quad
and} \\ 
\prob(\S_7 \given \S_1) \ge \prob(\S_5) \ge \prob(\S_2) \ge
n^{-\Theta(\epsilon_1)}.
\end{split}
\end{equation}
Let $\S_8$ be the event that for every node at even distance less than $2c$
from the root of $\TT_5$, the number of children of that node which in turn
have exactly $j$ children is in $z_{i,j}(1 \pm 3000 \Gamma_i)$.  By
Lemma~\ref{lemma:survive2}, Chernoff's bound and the union bound, one can find that $\prob(\S_8)
\ge 1 - n^{-\omega(1)}$. This, with~(\ref{eq:g4}), and the fact that
$\prob(\S_1) \ge n^{-\Theta(\epsilon_1)}$ (which follows from
Lemmas~\ref{lemma:survive2}~and~\ref{fact:f1}), implies
\begin{equation} \label{eq:g5}
\begin{split}
\prob(\S_6 \given \S_1 \wedge \S_8) \in \prob(\S_6 \given \S_1) (1 \pm
n^{-\omega(1)}) \text{\quad and } \\
\prob(\S_7 \given \S_1 \wedge \S_8) \in \prob(\S_7 \given \S_1) (1 \pm
n^{-\omega(1)}).
\end{split}
\end{equation}

To conclude the proof, note that if we condition on $\S_1 \wedge \S_8$, $\TT_5$
is a random tree which is isomorphic to a tree that satisfies the same
properties as $\TT_2$. Hence, conditioned on $\S_1 \wedge \S_8$, the
probability of $\S_6$ is a weighted average of elements in $\bigcup_{\TT_2'}
\{p_{\TT_2'}(n^{-\epsilon_1})\}$, and the probability of $\S_7$ is a weighted
average of elements in $\bigcup_{\TT_2'} \{ n^{\epsilon_1}
P_{\TT_2'}(n^{-\epsilon_1}) \}$, where $\bigcup_{\TT_2'}$ ranges over trees
that satisfy the same properties as $\TT_2$.  By applying
Lemma~\ref{lemma:survive1} twice we get that $p_{\TT_2'}(n^{-\epsilon_1}) \in
p_{\TT_2}(n^{-\epsilon_1}) ( 1 \pm 11 \Gamma_i \gamma_i )$ for every tree
$\TT_2'$ that satisfies the same properties as $\TT_2$.  Therefore, $\prob(\S_6
\given \S_1 \wedge \S_8) \in p_{\TT_2}(n^{-\epsilon_1}) (1 \pm 11 \Gamma_i
\gamma_i)$. Hence, $p_{\TT_2}(n^{-\epsilon_1}) \in \prob(\S_6 \given \S_1
\wedge \S_8) (1 \pm 12 \Gamma_i \gamma_i)$.  A similar argument shows that
$P_{\TT_2}(n^{-\epsilon_1}) \in n^{-\epsilon_1} \prob(\S_7 \given \S_1 \wedge
\S_8) (1 \pm 12 \Gamma_i \gamma_i)$.  This, together
with~(\ref{eq:g3})~and~(\ref{eq:g5}), implies
\begin{eqnarray*}
p_{\TT_2}(n^{-\epsilon_1}) \in \frac{\phi((i+1)n^{-\epsilon_1})}{
\phi(in^{-\epsilon_1})} (1 \pm 13 \Gamma_ i \gamma_i) \text{\quad and } \\
P_{\TT_2}(n^{-\epsilon_1}) \in \frac{\Phi((i+1)n^{-\epsilon_1}) -
\Phi(in^{-\epsilon_1})}{\phi(in^{-\epsilon_1})} (1 \pm 13 \Gamma_i \gamma_i).
\end{eqnarray*}
These estimates, together Lemma~\ref{lemma:survive1}, give
Lemma~\ref{lemma:survive}.

\section{Survival and the process} \label{sec:6}
%
In this section we relate the main result of the previous section to the
process.
Fix for the rest of the section an integer $0 \le i < I$.  For an integer $1
\le j \le 5$ and an edge $f \in \NotTrav_i$, let $X''_{i,j}(f)$ be the set of
all $G \in X_{i,j}(f)$ such that $G \subseteq \M_i \cup \BigBite_{i+1}$.
For an integer $c \ge 1$ and an edge $f \in \NotTrav_i$, we define a finite,
rooted, labeled tree $\TT_c(f)$; to do so, we first define another tree
$\TT'_c(f)$ and then alter it to obtain $\TT_c(f)$.  Let $\TT'_c(f)$ be the
finite, rooted, labeled tree with the following four properties: first, every
leaf in the tree is at distance $2c$ from the root; second, the root is labeled
with the edge $f$; third, if a non-leaf node $\nu$ at even distance from the
root is labeled with an edge $g$, then its set of children is the set $\{\nu_G
: G \in \bigcup_{1 \le j \le 5} X''_{i,j}(g)\}$, where the label of $\nu_G$ is
the graph $G$; fourth, if a node $\nu$ at odd distance from the root is labeled
with a graph $G$, then its set of children is the set $\{\nu_g  : g \in G \cap
\BigBite_{i+1}\}$, where the label of $\nu_g$ is the edge $g$.
Let $\TT_c(f)$ be obtained by removing subtrees from $\TT'_c(f)$ as follows:
for every non-leaf, non-root node $\nu$ at even distance from the root, if
$\nu$ has a child labeled $G$ and a grandparent labeled $g \in G$, then remove
the subtree rooted at the child labeled $G$.

Let $\nu_0$ be a node labeled $g_0$ at even distance from the root of
$\TT_c(f)$.  Define the event that $\nu_0$ survives as follows. If $\nu_0$ is a
leaf then $\nu_0$ survives by definition.  Otherwise, $\nu_0$ survives if and
only if for every child $\nu_1$ of $\nu_0$, the following holds: if for every
child $\nu_2$ of $\nu_1$, labeled $g_2$, we have $g_2 \in \Bite_{i+1}$, and in
case $g_0 \in \Bite_{i+1}$ we also have that the birthtime of $g_2$ is less
than the birthtime of $g_0$, then $\nu_1$ has a child that does not survive.
For $f \in \NotTrav_i$, let $\S_c(f)$ be the event that the root of $\TT_c(f)$
survives.  For $F \subseteq \NotTrav_i$, let $\S_c(F) \deq \bigwedge_{f \in F}
\S_c(f)$, and let $\I_c(F)$ be the event that every two distinct nodes at even
distances from the roots of the trees in the forest $\{\TT_c(f) : f \in F\}$
have two distinct labels. 
\begin{lemma} \label{lemma:conn}
Let $c \ge 1$ be an odd integer. Let $F \subseteq \NotTrav_i$ and assume that
$\M_i \cup F$ is $K_4$-free. Then: (i) assuming $|F|=1$, $\S_c(F) \implies
\text{ $\M_{i+1} \cup F$ is $K_4$-free} \implies \S_{c+1}(F)$; (ii) assuming
$|F| \ge 2$, $\I_c(F) \wedge \S_c(F) \implies \text{ $\M_{i+1} \cup F$ is
$K_4$-free}$.  
\end{lemma}
\begin{proof}
Let $c \ge 1$ be an odd integer, let $F \subseteq \NotTrav_i$ and assume that
$\M_i \cup F$ is $K_4$-free.  The second item follows directly from the first
item.
So it remains to prove the first item. For that purpose, assume for the rest of
the proof that $F = \{f\}$.  We need the following claim.
\begin{claim} \label{claim:sec:4:1}
Let $b \ge 2$ be an even integer.  Let $\nu_0$ be a node labeled $g_0$ at
height $2b$ in $\TT_c(f)$ or in $\TT_{c+1}(f)$. If $\nu_0$ doesn't survive then
$\M_{i+1} \cup \{g_0\}$ is not $K_4$-free.
\end{claim}
\begin{proof}
The proof is by induction on $b$.  We start with a general setup that applies
both to the base case and to the induction step. Let $b \ge 2$ be an even
integer. Let $\nu_0$ be a node labeled $g_0$ at height $2b$ in $\TT_c(f)$ or in
$\TT_{c+1}(f)$. Assume that $\nu_0$ doesn't survive.
We show that there is a node $\nu_1$ labeled $G_1$, which is a child of
$\nu_0$, such that $G_1 \subseteq \M_{i+1}$. This will give us that $\M_{i+1}
\cup \{g_0\}$ is not $K_4$-free.
Since $\nu_0$ doesn't survive we have that there is a node $\nu_1$ labeled
$G_1$, which is a child of $\nu_0$, such that for every child $\nu_2$ of
$\nu_1$, letting $g_2$ be the label of $\nu_2$, the following two properties
hold: first, $g_2 \in \Bite_{i+1}$ and if $g_0 \in \Bite_{i+1}$ then the
birthtime of $g_2$ is less than the birthtime of $g_0$; second, $\nu_2$
survives. 
Fix such a child $\nu_1$ of $\nu_0$. It remains to show that for every child
$\nu_2$ of $\nu_1$, letting $g_2$ be the label of $\nu_2$, the fact that
$\nu_2$ survives implies $\M_{i+1} \cup \{g_2\}$ is $K_4$-free, as we already
know that $g_2 \in \Bite_{i+1}$. This will give us that $G_1 \subseteq
\M_{i+1}$.  So let us fix such a child $\nu_2$ of $\nu_1$. 
To show that $\M_{i+1} \cup \{g_2\}$ is $K_4$-free we need to show that for
every $G_3 \in \bigcup_{1 \le j \le 5} X''_{i,j}(g_2)$, either there is an edge
in $G_3$ whose birthtime is larger than that of $g_2$, or otherwise $G_3
\nsubseteq \M_{i+1}$. Fix a graph $G_3 \in \bigcup_{1 \le j \le 5}
X''_{i,j}(g_2)$. Then either $G_3$ is a label of a child of $\nu_2$ or $g_0 \in
G_3$, and this follows from the definition of $\TT_c(f)$ and $\TT_{c+1}(f)$. If
$g_0 \in G_3$ then the birthtime of $g_0$ is larger than the birthtime of $g_2$
and we are done.  So assume that $G_3$ is a label of a child $\nu_3$ of
$\nu_2$.  This is where the arguments for the base case and the induction step
differ.

For the base case, assume that $b=2$. Then every child of $\nu_3$ is a leaf.
Since a leaf survives by definition, the fact that $\nu_2$ survives implies
that $\nu_3$ either has a child whose label is not in $\Bite_{i+1}$, in which
case $G_3 \nsubseteq \M_{i+1}$ as needed, or $\nu_3$ has a child whose label
has a birthtime larger than that of $g_2$, as needed.  For the induction step,
assume that $b \ge 4$ and that the claim holds for $b-2$. If $\nu_2$ survives
then either $\nu_3$ has a child whose label is not in $\Bite_{i+1}$, in which
case $G_3 \nsubseteq \M_{i+1}$ as needed; or $\nu_3$ has a child whose label
has a birthtime larger than that of $g_2$, as needed; or $\nu_3$ has a child
that doesn't survive, in which case, by the induction hypothesis, for some $g_4
\in G_3$, $\M_{i+1} \cup \{g_4\}$ is not $K_4$-free, and so $G_3 \nsubseteq
\M_{i+1}$ as needed.
\end{proof}

Since $c+1$ is even, it follows from the claim above that $\M_{i+1} \cup F
\text{ is $K_4$-free} \implies \S_{c+1}(F)$. Since $c$ is odd, it also follows
from the claim above and from the definition of $\S_c(F)$, that if $\S_c(F)$
holds then the following holds: for every node $\nu$ whose label is $G$ and
which is a child of the root of $\TT_c(f)$, $G \nsubseteq \M_{i+1}$. Since
$\M_i \cup F$ is $K_4$-free and since every graph $G \in \bigcup_{1 \le j \le
5} X''_{i,j}(f)$ is a label of a child of the root of $\TT_c(f)$, we get that
$\S_c(F) \implies \M_{i+1} \cup F \text{ is $K_4$-free}$. 
\end{proof}

For an edge $f \in \NotTrav_i$ and a set $R \subseteq [n]$, let $\TT_c(f,R)$ be
obtained by removing subtrees from $\TT_c(f)$ as follows: for every child $\nu$
of the root of $\TT_c(f)$, if $\nu$ is labeled with a graph that shares at
least three vertices with $R$, then remove the subtree rooted at $\nu$. Define
the event that a node         at even distance from the root of $\TT_c(f,R)$
survives exactly as it was defined for such a node in $\TT_c(f)$.  
Let $\S_c(f, R)$ be the event that the root of $\TT_c(f, R)$ survives.  For two
disjoint graphs $F_1, F_2 \subseteq \NotTrav_i$, let $\S_c(F_1, F_2, R)$ be the
event $\big[\bigwedge_{f \in F_1} \S_c(f)\big] \wedge \big[\bigwedge_{f \in
F_2} \S_c(f, R)\big]$, and let $\I_c(F_1, F_2, R)$ be the event that every two
distinct nodes at even distances from the roots of the trees in the forest
$\{\TT_c(f) : f \in F_1\} \cup \{\TT_c(f, R) : f \in F_2\}$ have two distinct
labels.
\begin{lemma} \label{lemma:conn2}
Let $c \ge 1$ be an odd integer. Let $F_1, F_2 \subseteq \NotTrav_i$ be two
disjoint graphs and assume that $\M_i \cup F_1 \cup F_2$ is $K_4$-free. Let $R
\subseteq [n]$.  Then: (i) $\I_c(F_1 \cup F_2) \wedge \S_c(F_1 \cup F_2)
\implies \I_c(F_1, F_2, R) \wedge \S_c(F_1, F_2, R)$; (ii) $\I_c(F_1, F_2, R)
\wedge \S_c(F_1, F_2, R)$ implies that $\M_{i+1} \cup F_1$ is $K_4$-free, and
for every edge $f \in F_2$, if $|X_{i+1,0}(f)| > 0$ (meaning $\M_{i+1} \cup
\{f\}$ is not $K_4$-free), then for every $G \in X_{i+1,0}(f)$, $G$ shares at
least three vertices with $R$.
\end{lemma}
\begin{proof}
The first item follows from the definition of the underlying events.  The
second item can be proved using the same argument used in the proof of
Lemma~\ref{lemma:conn}, and in particular, using Claim~\ref{claim:sec:4:1}.
\end{proof}

Fix for the rest of the section a graph $F \subseteq \NotTrav_i$, with $1 \le
|F| = a_1 + a_2 \le 3$, and such that $\M_i \cup F$ is $K_4$-free.  Also, fix
an integer $c \in \epsilon_3^2 \epsilon_2^{-1} \pm 1$, and note that $c$ is a
sufficiently large constant. 
Let $\E_0$ be the event that $|F \cap \Bite_{i+1}| = a_1$.
Let $\E_1$  be the event that for every $f \in F$, the set of children of every
non-leaf node at even distance from the root of $\TT_c(f)$ can be partitioned
to $5$ sets of children, where the $j$th set satisfies the following: it
consists of nodes whose labels have exactly $j$ edges in $\BigBite_{i+1}$, and
it has size in $z_{i,j} ( 1 \pm 3000 \Gamma_i)$.  For brevity, set $\E_2 =
\I_c(F)$.

The next lemma is the main result of this section. 
\begin{lemma} \label{lemma:survive:process}
Assume that $\M_i$ and $\BIGBite_{i+1}$ are given so that $\C_i \wedge \D_i$
holds. Then 
\begin{eqnarray*} \label{eq:sec:4:1}
\prob(\S_c(F) \given \E_0) \in
\bigg(\frac{\Phi((i+1)n^{-\epsilon_1}) -
\Phi(in^{-\epsilon_1})}{n^{-\epsilon_1} \phi(in^{-\epsilon_1})}\bigg)^{a_1}
\bigg(\frac{\phi((i+1)n^{-\epsilon_1})}{\phi(in^{-\epsilon_1})}\bigg)^{a_2} (1
\pm 90 \Gamma_i \gamma_i) 
\end{eqnarray*}
and 
\begin{eqnarray*} \label{eq:sec:4:2}
\prob(\E_2 \wedge \S_c(F) \given \E_0) \in \prob(\S_c(F) \given \E_0) (1 \pm
o(\Gamma_i \gamma_i)),
\end{eqnarray*}
where the probabilities are both over the choice of $\BigBite_{i+1}$,
$\Bite_{i+1}$ and the choice of the birthtimes of the edges in $\Bite_{i+1}$.
\end{lemma}

Let us prove Lemma~\ref{lemma:survive:process}.  To this end, assume for the
rest of the section that $\M_i$ and $\BIGBite_{i+1}$ are given so that $\C_i
\wedge \D_i$ holds.  In the two subsections below we will prove that
\begin{eqnarray}
\label{eq:sec:6:1} \prob(\E_1 \given \E_0) &\ge& 1 - n^{-\omega(1)} \text{\quad
and } \\
\label{eq:sec:6:2} \prob(\E_2 \given \E_0) &\ge& 1 - n^{-\Theta(\epsilon_3)}.
\end{eqnarray}
Here we use these estimates to prove the lemma.  

We start by obtaining an estimate for $\prob(\S_c(F) \given \E_0)$.  First,
observe that~(\ref{eq:sec:6:1})~and~(\ref{eq:sec:6:2}), together with
Lemma~\ref{fact:f1} and the fact that $\epsilon_1$ is sufficiently small with
respect to $\epsilon_3$, imply that
\begin{eqnarray*} \label{eq:a1}
\prob(\E_1 \wedge \E_2 \given \E_0) \ge 1 - o(\Gamma_i \gamma_i).
\end{eqnarray*}
Second, note that if we condition on $\E_0 \wedge \E_1 \wedge \E_2$, then every
tree in $\{\TT_c(f) : f \in F\}$ satisfies the same properties that are
satisfied by the tree $\TT_1$, the tree that was studied in the previous
section, and the events in $\{ \S_c(f) : f \in F \}$ are mutually independent.
Hence, under this condition we can use the main result of the previous section,
Lemma~\ref{lemma:survive}, to find that
\begin{eqnarray*} 
\prob(\S_c(F) \given \E_0 \wedge \E_1 \wedge \E_2) \in
\bigg(\frac{\Phi((i+1)n^{-\epsilon_1}) -
\Phi(in^{-\epsilon_1})}{n^{-\epsilon_1} \phi(in^{-\epsilon_1})}\bigg)^{a_1}
\bigg(\frac{\phi((i+1)n^{-\epsilon_1})}{\phi(in^{-\epsilon_1})}\bigg)^{a_2} (1
\pm 80 \Gamma_i \gamma_i).
\end{eqnarray*}
Third, observe that $\prob(\S_c(F) \given \E_0 \wedge \E_1 \wedge \E_2) =
\Omega(1)$.  (Indeed, a sufficient condition for $\S_c(F)$ is that for every $f
\in F$, for every child $\nu_1$ of the root of $\TT_c(f)$ there is a child
$\nu_2$ whose label is not in $\Bite_{i+1}$. Assuming $\E_1 \wedge \E_2$ this
event occurs with probability $\Omega(1)$.) Since $\prob(\S_c(F) \given \E_0) =
\prob(\E_1 \wedge \E_2 \given \E_0) \prob(\S_c(F) \given \E_0 \wedge \E_1
\wedge \E_2) + O(\prob(\neg (\E_1 \wedge \E_2) \given \E_0))$, the above three
facts give the desired estimate for $\prob(S_c(F) \given \E_0)$.

To obtain an estimate for $\prob(\E_2 \wedge \S_c(F) \given \E_0)$ and complete
the proof, simply note that
\begin{eqnarray*}
\prob(\S_c(F) \given \E_0) &\ge& \prob(\E_2 \wedge \S_c(F) \given \E_0) \\
&\ge& \prob(\E_1 \wedge \E_2 \wedge \S_c(F) \given \E_0) \\
&\ge& \prob(\E_1 \wedge \E_2 \given \E_0) \cdot \prob(\S_c(F) \given \E_0
\wedge  \E_1 \wedge \E_2),
\end{eqnarray*}
and apply our findings from above.

\subsection{Proof of~(\ref{eq:sec:6:1})} \label{sec:6:1}
%
Since clearly $\prob(\E_0) \ge n^{-\Theta(1)}$, it suffices to prove that
$\prob(\E_1) \ge 1 - n^{-\omega(1)}$.  Let $f \in F$, let $\nu$ be a non-leaf
node labeled $g$ at even distance from the root of $\TT_c(f)$, and let $1 \le j
\le 5$. Note that by the union bound it is enough to prove that each of the
following two properties occurs with probability at least $1 - n^{-\omega(1)}$:
first, the number of children of $\nu$ which are labeled with a graph $G \in
X''_{i,j}(g)$ is equal to $|X''_{i,j}(g)|$, up to an additive factor of $(\ln
n)^{O(1)}$; second, $|X''_{i,j}(g)| \in z_{i,j}(1 \pm 2999 \Gamma_i)$.

To show that the first property occurs with probability at least $1 -
n^{-\omega(1)}$, recall the definition of $\TT_c(f)$ and observe that it
suffices to show that with probability at least $1 - n^{-\omega(1)}$, for every
three vertices $v_1, v_2, v_3 \in [n]$, there are at most $(\ln n)^{O(1)}$
other vertices in $[n]$ that are adjacent in $\Trav_i \cup \BigBite_{i+1}$
simultaneously to $v_1, v_2$ and $v_3$.  Indeed, by the fact that
$\BigBite_{i+1} \subseteq \BIGBite_{i+1}$ and by~(C3), the above occurs with
probability $1$.

Next, we show that the second property occurs with probability at least $1 -
n^{-\omega(1)}$.
For that we assume that either $i \ge 1$, or else $j = 5$, since otherwise
trivially $|X''_{i,j}(f)| = z_{i,j} = 0$ and we are done.
Let $\{G_l' : l \in L\}$ be the set $X'_{i,j}(g)$. By~(D2)
we have $|L| = |X'_{i,j}(g)| \in n^{(\epsilon_3 - 2/5)j} x_{i,j} (1 \pm 2000
\Gamma_i)$.  Let $\{G_l : l \in L\}$ be the family (potentially a multiset) for
which it holds that $G_l = G_l' \cap \BIGBite_{i+1}$ for every $l \in L$.
Consider the binomial random graph $G(n,p)$ with $p = n^{\epsilon_2 -
\epsilon_3}$, and let $W$ be as defined at the beginning of
Section~\ref{sec:3}. Note that $\expec(W) = n^{(\epsilon_2 - \epsilon_3) j} |L|
\in z_{i,j} (1 \pm 2000 \Gamma_i)$. Also note that $W$ has the same
distribution as $|X''_{i,j}(g)|$. It remains to argue that the probability that
$W$ deviates from its expectation by more than $999 \Gamma_i z_{i,j} \ge
\expec(W)^{0.9} = n^{\Omega(\epsilon_2)}$ is at most $n^{-\omega(1)}$.  To do so, note
that if $1 \le |G| \le j$ then $|L_G| \le (\ln n)^{O(1)}$, and that this
follows from~(C3).  Apply Theorem~\ref{thm:vu} with $\EE_0 = \expec(W)$, $\EE_k
= \exp((2j-k)\sqrt{\ln n})$ for $1 \le k \le j$, and $\lambda = (\ln n)^2$.

\subsection{Proof of~(\ref{eq:sec:6:2})} \label{sec:6:2}
%
Say that a sequence $(G_l)_{l=1}^m$ is bad, if the following properties hold:
\begin{itemize}
\item $1 \le m \le 2c$; 
\item for all $1 \le l \le m$, $G_l \in \bigcup_{1 \le j \le 5} X'_{i,j}(g)$
for some $g \in F \cup \bigcup_{k<l} (G_k \cap \BIGBite_{i+1})$; 
\item for all $1 \le l < m$, $G_l \cap \BIGBite_{i+1}$ shares no edge with $F
\cup \bigcup_{k<l} (G_k \cap \BIGBite_{i+1})$;
\item one of the following holds:
\begin{itemize}
\item $G_m \cap \BIGBite_{i+1}$ shares at least $1$ edge but not all edges with
$F \cup \bigcup_{k<m} (G_k \cap \BIGBite_{i+1})$;
\item $G_m \cap \BIGBite_{i+1}$ shares no edge with $F \cup \bigcup_{k<m} (G_k
\cap \BIGBite_{i+1})$, and there is a vertex outside of the vertex set of $G_m$
that is adjacent in $\Trav_i \cup \BIGBite_{i+1}$ to at least three vertices of
$G_m$;
\item $G_m \cap \BIGBite_{i+1}$ shares no edge with $F \cup \bigcup_{k<m} (G_k
\cap \BIGBite_{i+1})$. Moreover, let $g_m$ be such that $G_m \in \bigcup_{1 \le
j \le 5} X'_{i,j}(g_m)$. Then there is a vertex of $G_m$ that is not a vertex
of $g_m$, and which is adjacent in $\Trav_i \cup \BIGBite_{i+1} \cup F$ to at
least three vertices of $F \cup \bigcup_{k<m} G_k$.
\end{itemize}
\end{itemize}

Note that for every bad sequence it holds that its members are all contained in
$\M_i \cup \BIGBite_{i+1}$.
Let $\E_3$ be the event that there is no bad sequence whose members are all
contained in $\M_i \cup \BigBite_{i+1}$.  The next two lemmas
imply~(\ref{eq:sec:6:2}).
\begin{lemma}
$\E_3 \implies \E_2$.
\end{lemma}
\begin{proof}
We prove the contrapositive. Assume $\neg \E_2$ holds and consider the forest
$\{\TT_c(f): f \in F\}$.  Then for some $1 \le m \le 2c$ there is a sequence
$(\nu_l)_{l=1}^m$ of nodes at odd distances from the roots in the forest such
that, denoting by $G_l$ the label of $\nu_l$, the following holds: for all $1
\le l \le m$, $\nu_l$ is either a child of a root in the forest, or a
grandchild of some $\nu_k$ with $k < l$; furthermore, for all $1 \le l < m$,
$G_l \cap \BigBite_{i+1}$ shares no edge with $F \cup \bigcup_{k < l} (G_k \cap
\BigBite_{i+1})$, while $G_m \cap \BigBite_{i+1}$ shares at least $1$ edge with
$F \cup \bigcup_{k<m} (G_k \cap \BigBite_{i+1})$.
Let $g_l$ be the label of the parent of $\nu_l$.  Note that for all $1 \le l
\le m$, $g_l \in F \cup \bigcup_{k<l} (G_k \cap \BigBite_{i+1})$, $G_l \in
\bigcup_{1 \le j \le 5} X'_{i,j}(g_l)$, and $G_l \cap \BigBite_{i+1} = G_l \cap
\BIGBite_{i+1}$.  Hence, every non-empty prefix of $(G_l)_{l=1}^m$ satisfies
the first three properties of a bad sequence. By definition, the members of
$(G_l)_{l=1}^m$ are all contained in $\M_i \cup \BigBite_{i+1}$, and so in
order to conclude that $\neg \E_3$ holds it remains to show that some non-empty
prefix of $(G_l)_{l=1}^m$ satisfies the fourth property of a bad sequence. 

Since $G_m \cap \BigBite_{i+1}$ shares at least $1$ edge with $F \cup
\bigcup_{k<m} (G_k \cap \BigBite_{i+1})$, we may        assume that $G_m \cap
\BigBite_{i+1}$ shares all of its edges with $F \cup \bigcup_{k<m} (G_k \cap
\BigBite_{i+1})$, since otherwise $(G_l)_{l=1}^m$ satisfies the fourth property
of a bad sequence and we are done.

Suppose that $G_m \cap \BigBite_{i+1} \subseteq F$. By assumption, $\M_i \cup
F$ is $K_4$-free and so since $G_m \subseteq \M_i \cup \BigBite_{i+1}$,
we must have that $g_m \notin F$.  Hence, there exists
$1 \le m' < m$ such that $g_m \in G_{m'}$. We claim that $(G_l)_{l=1}^{m'}$
satisfies the fourth property of a bad sequence. Indeed, let $f \in G_m \cap
\BigBite_{i+1} \subseteq F$. By the definition of the trees in the forest,
$g_{m'} \notin G_m$ and so $f \ne g_{m'}$. 
Also, $G_{m'} \cap \BigBite_{i+1}$ shares no edge with $F$ and so $f \notin
G_{m'}$.
It follows that there is a vertex of $G_{m'}$ (more accurately, a vertex of
$g_m$) that is not a vertex of $g_{m'}$, and which is adjacent in $\Trav_i \cup
\BIGBite_{i+1}$ to at least three vertices of $F \cup \bigcup_{k<m'} G_k$
(these three vertices being the two vertices of $g_{m'}$ and one vertex of
$f$).

Suppose that $G_m \cap \BigBite_{i+1} \nsubseteq F$. Then there exists $1 \le
m' < m$ such that $G_m \cap \BigBite_{i+1}$ and $G_{m'} \cap \BigBite_{i+1}$
share some edge $g$, and such that $m'$ is maximal with respect to that
property.  If $\nu_m$ and $\nu_{m'}$ are siblings, then it is clear that there
is a vertex outside of the vertex set of $G_{m'}$ that is adjacent in $\Trav_i
\cup \BIGBite_{i+1}$ to at least three vertices of $G_{m'}$.  If $\nu_m$ is a
grandchild of $\nu_{m'}$, then since by the definition of the trees in the
forest $g_{m'} \notin G_m$, we have that there is a vertex outside of the
vertex set of $G_{m'}$ that is adjacent in $\Trav_i \cup \BIGBite_{i+1}$ to at
least three vertices of $G_{m'}$. Therefore, if $\nu_m$ and $\nu_{m'}$ are
siblings or if $\nu_m$ is a grandchild of $\nu_{m'}$, then $(G_l)_{l=1}^{m'}$
satisfies the fourth property of a bad sequence, and we are done.
So assume that $\nu_m$ and $\nu_{m'}$ are not siblings, and that $\nu_m$ is not
a grandchild of $\nu_{m'}$. Then since $(G_l)_{l=1}^m$ satisfies the first
three properties of a bad sequence, we have that $g_m \ne g_{m'}$ and $g_m
\notin G_{m'}$.
Now, we have two cases: either $g_m \in F \cup \bigcup_{k < m'} (G_k \cap
\BigBite_{i+1})$, or $g_m \in \bigcup_{m' \le k < m} (G_k \cap
\BigBite_{i+1})$. 
If the first case holds, since $g_m \ne g_{m'}$, we have that there is a vertex
of $G_{m'}$ (more accurately, a vertex of $g$) that is not a vertex of
$g_{m'}$, and which is adjacent in $\Trav_i \cup \BIGBite_{i+1} \cup F$ to at
least three vertices of $F \cup \bigcup_{k<m'} G_k$ (these three vertices being
the two vertices of $g_{m'}$ and one vertex of $g_m$).
If the second case holds, since $g_m \notin G_{m'}$, we have that $g_m \in
\bigcup_{m' < k < m} (G_k \cap \BigBite_{i+1})$. In that case, $g_m \in G_{m''}
\cap \BigBite_{i+1}$, for some $m' < m'' < m$.  Since $g \in G_m \cap
\BigBite_{i+1}$ and $G_m \in \bigcup_{1 \le j \le 5} X'_{i,j}(g_m)$, by the
definition of the trees in the forest, $g \ne g_{m''}$. Also, by the maximality
of $m'$, $g \notin G_{m''}$. Hence, $g$ has a vertex outside of the vertex set
of $G_{m''}$. Note that $g \in F \cup \bigcup_{k < m''} (G_k \cap \BigBite_{i+1})$. It follows that
there is a vertex of $G_{m''}$ (more accurately, a vertex of $g_m$) that is not
a vertex of $g_{m''}$, and which is adjacent in $\Trav_i \cup \BIGBite_{i+1}$
to at least three vertices of $F \cup \bigcup_{k < m''} G_k$ (these three
vertices being the two vertices of $g_{m''}$ and one vertex of $g$).
\end{proof}

\begin{lemma}
$\prob(\E_3 \given \E_0) \ge 1 - n^{-\Theta(\epsilon_3)}$.
\end{lemma}
\begin{proof}
Fix $1 \le m \le 2c$ and note that $m = O(1)$. By the union bound, it is enough
to show that conditioned on $\E_0$, the expected number of bad sequences of
length $m$ whose members are all contained in $\M_i \cup \BigBite_{i+1}$ is at
most $n^{-\Theta(\epsilon_3)}$.

Say that a sequence $(G_l)_{l=1}^{m-1}$ is almost-bad if there exists $G_m$
such that $(G_l)_{l=1}^m$ is bad. 
We first claim that conditioned on $\E_0$, the expected number of almost-bad
sequences whose members are all contained in $\M_i \cup \BigBite_{i+1}$ is at
most $n^{O(\epsilon_2 m)}$, which is at most $n^{O(\epsilon_3^2)}$, since $m
\le 2 c \le 4 \epsilon_3^2 \epsilon_2^{-1}$.  This claim follows from the
definition of a bad sequence, the fact that~(D2) holds (specifically the fact
that $|X'_{i,j}(f)| \le n^{\epsilon_3 j + o(1)}$ for all $1 \le j \le 5$ and
all $f \in \NotTrav_i$), and since the probability that $G \in X'_{i,j}(f)$ is
contained in $\M_i \cup \BigBite_{i+1}$ is $n^{(\epsilon_2 - \epsilon_3)j}$.

Given an almost-bad sequence $(G_l)_{l=1}^{m-1}$, the number of graphs $G_m
\subseteq \M_i \cup \BIGBite_{i+1}$ that can be concatenated to this sequence
so that it becomes bad is at most $(\ln n)^{O(1)}$, and this follows from the
definition of a bad sequence together with~(C3)~and~(C4).  For every almost-bad
sequence $(G_l)_{l=1}^{m-1}$ and every graph $G_m \subseteq \M_i \cup
\BIGBite_{i+1}$ that can be concatenated to this sequence so that it becomes
bad, there is an edge in $G_m \cap \BIGBite_{i+1}$ which does not belong to
$F$, nor to any of the graphs in the almost-bad sequence.  Hence, conditioned
on the event that every member of a given almost-bad sequence is contained in
$\M_i \cup \BigBite_{i+1}$, and conditioned on $\E_0$, the probability that a
given $G_m$ as above is contained in $\M_i \cup \BigBite_{i+1}$ is at most
$n^{\epsilon_2 - \epsilon_3}$.  It follows that the expected number of bad
sequences of length $m$ is at most $n^{O(\epsilon_3^2)} \cdot (\ln n)^{O(1)}
\cdot n^{\epsilon_2 - \epsilon_3}$, which is at most $n^{-\Theta(\epsilon_3)}$.
\end{proof}

\section{Supporting lemmas} \label{sec:4}
%
In this section we prove two supporting lemmas that will be used in the next
section, where we prove our main lemma. Fix for the rest of the section an
integer $0 \le i < I$.  We start with the following lower bound on the
probability of $\C_i$.
\begin{lemma} \label{lemma:sec:4:1}
$\prob(\C_i) \ge 1 - n^{\omega(1)}$.
\end{lemma}
\begin{proof}
We argue that every property that is asserted to hold by $\C_i$ occurs with
probability at least $1 - n^{-\omega(1)}$.  The key observation here is that
$\Trav_i \cup \BIGBite_{i+1}$ is the binomial random graph $G(n,p)$, for some
$p = \Theta( n^{\epsilon_3 - 2/5})$ that we fix for the rest of the proof.
With that observation at hand, we continue as follows.  A standard application
of Chernoff's bound and the union bound shows that~(C1),~(C2)~and~(C3), each
occurs with probability at least $1 - n^{-\omega(1)}$.  Theorem~\ref{thm:vu}
and the union bound easily implies that~(C4)~and~(C5), each occurs with
probability at least $1 - n^{-\omega(1)}$.  Theorem~\ref{thm:deletion} and the
union bound easily implies that~(C6) occurs with probability at least $1 -
n^{-\omega(1)}$.  To complete the proof, we need to show that~(C7)~and~(C8),
each occurs with probability at least $1 - n^{-\omega(1)}$.  For that, we need
the following claim.

\begin{claim} \label{claim:sec:6}
Let $R \subseteq [n]$ be a set of $r$ vertices, where $s - o(s) \le r \le s$.
Let $Q$ be a set containing at most $q$ paths of length two in $\binom{R}{2}$.
With probability at least $1 - n^{-\omega(s)}$, the following holds for every
$\M \subseteq G(n,p)$, assuming the maximum degree in $\M \cap \binom{R}{2}$ is
at most $n^{1.1/5}$: the number of paths in $Q$ that are contained in $\M \cap
\binom{R}{2}$ is at most $O(q n^{-3.99/5} + n^{4.2/5})$.
\end{claim}
\begin{proof}
The expected number of paths in $Q$ that are contained in $G(n,p)$ is at most
$q p^2 \le q n^{-3.99/5}$. Hence, by Theorem~\ref{thm:deletion}, with
probability at least $1 - n^{-\omega(s)}$, there is a set $E_0 \subseteq
G(n,p)$ of size at most $n^{3.1/5}$, such that $G(n,p) \setminus E_0$ contains
fewer than $2 q n^{-3.99/5}$ paths from $Q$.  Moreover, for every $\M \subseteq
G(n,p)$, assuming the maximum degree in $\M \cap \binom{R}{2}$ is at most
$n^{1.1/5}$, every edge in $E_0$ belongs to at most $2 n^{1.1/5}$ paths of
length two in $\M \cap \binom{R}{2}$.  Therefore, with the desired probability,
for every $\M \subseteq G(n,p)$, assuming the maximum degree in $\M \cap
\binom{R}{2}$ is at most $n^{1.1/5}$, the number of paths in $Q$ that are
contained in $\M \cap \binom{R}{2}$ is at most $2 q n^{-3.99/5} + |E_0| \cdot 2
n^{1.1/5} = O( q n^{-3.99/5} + n^{4.2/5})$.  
\end{proof}

{\bf Property~(C7).}
Fix a set $R \subseteq [n]$ of $r$ vertices, where $s - o(s) \le r \le s$.  To
show that~(C7) occurs with probability at least $1 - n^{-\omega(1)}$, it is
enough to show that $R$ satisfies the first assertion of~(C7) with probability
at least $1 - n^{-\omega(s)}$ and the second assertion of~(C7) with probability
at least $1 - n^{-\omega(s)}$.  

We begin with the first assertion.  We show that with probability at least $1 -
n^{-\omega(s)}$, for every $\M \subseteq G(n,p)$, assuming the maximum degree
in $\M \cap \binom{R}{2}$ is at most $n^{1.1/5}$, there are at most
$O(n^{4.2/5})$ edges $g \in \binom{R}{2}$ for which there exists  a graph $G
\in X_{0,5}(g)$, with $G \subseteq \M$, which shares all four vertices with
$R$.  For that it is enough to show that with probability at least $1 -
n^{-\omega(s)}$, for every $\M \subseteq G(n,p)$, assuming the maximum degree
in $\M \cap \binom{R}{2}$ is at most $n^{1.1/5}$, there are at most
$O(n^{4.2/5})$ edges $g \in \binom{R}{2}$ whose two vertices form an
independent set in some $4$-cycle in $\M \cap \binom{R}{2}$.  Note that the
expected number of $4$-cycles in $G(n,p) \cap \binom{R}{2}$ is at most
$n^{4.2/5}$. 
Hence, by Theorem~\ref{thm:deletion}, with probability at least $1 -
n^{-\omega(s)}$, there is a set $E_0 \subseteq G(n,p)$ of size at most
$n^{3.1/5}$, such that $(G(n,p) \cap \binom{R}{2}) \setminus E_0$ contains
fewer than $2n^{4.2/5}$ $4$-cycles.  Also, every $4$-cycle has two independent
sets of size two. Hence, with probability at least $1 - n^{-\omega(s)}$, for
every $\M \subseteq G(n,p)$, assuming the maximum degree in $\M \cap
\binom{R}{2}$ is at most $n^{1.1/5}$, the number of edges $g \in \binom{R}{2}$
whose two vertices form an independent set in some $4$-cycle in $\M \cap
\binom{R}{2}$ is at most $4n^{4.2/5} + |E_0| \cdot 2n^{1.1/5} = O(n^{4.2/5})$. 

We continue with the second assertion.  We show that with probability at least
$1 - n^{-\omega(s)}$, for every $\M \subseteq G(n,p)$, assuming the maximum
degree in $\M \cap \binom{R}{2}$ is at most $n^{1.1/5}$, there is a set $R_0
\subseteq [n] \setminus R$ of at most $n^{0.99/5}$ vertices, such that there
are at most $O(n^{4.2/5})$ edges $g \in \binom{R}{2}$ for which there exists a
graph $G \in X_{0,5}(g)$, with $G \subseteq \M$, which shares exactly three
vertices with $R$ and one vertex with $[n] \setminus (R \cup R_0)$.
Start by exposing only the edges in $G(n,p) \setminus \binom{R}{2}$.  It is
easy to show that with probability at least $1 - n^{-\omega(s)}$, we can
partition the set of vertices $[n] \setminus R$ to two sets, $R_0$ and $R_1$,
where $R_0$ has size at most $n^{0.99/5}$, where every vertex in $R_1$ is
adjacent in $G(n,p)$ to at most $n^{2.02/5}$ vertices in $R$, and where the
number of vertices in $R_1$ which are adjacent in $G(n,p)$ to more than
$n^{1.01/5}$ vertices in $R$ is at most $n^{2/5}$.  We condition on the
occurrence of this event and continue by exposing  the edges in $G(n,p) \cap
\binom{R}{2}$.
For $\M \subseteq G(n,p)$, let $E_1(\M)$ be the set of edges $g \in
\binom{R}{2}$ for which there exists a graph $G \in X_{0,5}(g)$, with $G
\subseteq \M$, which shares exactly three vertices with $R$ and one vertex with
$R_1 = [n] \setminus (R \cup R_0)$. 
Note that $|E_1(\M)|$ is bounded by the number of paths of length two in $\M
\cap \binom{R}{2}$, whose three vertices are adjacent in $G(n,p)$ to a vertex
in $R_1$.  Before the second exposure, the number of possible paths of this
kind is at most 
\begin{eqnarray*}
3 \cdot |R_1| \cdot (n^{1.01/5})^3 + 3 \cdot n^{2/5} \cdot (n^{2.02/5})^3 \le
n^{8.1/5}. 
\end{eqnarray*}
Hence, by Claim~\ref{claim:sec:6}, with probability at least $1 -
n^{-\omega(s)}$, for every $\M \subseteq G(n,p)$, assuming the maximum degree
in $\M \cap \binom{R}{2}$ is at most $n^{1.1/5}$, of these possible paths, only
$O(n^{4.2/5})$ belong to $\M \cap \binom{R}{2}$.  So with probability at least
$1 - n^{-\omega(s)}$, for every $\M \subseteq G(n,p)$, assuming the maximum
degree in $\M \cap \binom{R}{2}$ is at most $n^{1.1/5}$, we have $|E_1(\M)| =
O(n^{4.2/5})$.  This completes the proof.

{\bf Property~(C8).}
Fix a set $R \subseteq [n]$ of $r$ vertices, where $s - o(s) \le r \le s$, and
a set $E$ of $O(n^{1/2})$ edges in $\binom{[n]}{2} \setminus \binom{R}{2}$. To
show that~(C8) occurs with probability at least $1 - n^{-\omega(1)}$, it is
enough to show that $R$ and $E$ satisfies the assertion of~(C8) with
probability at least $1 - n^{-\omega(s)}$.  Furthermore, it suffices to do so
conditioned on~(C3).
Let $E_1$ be the set of edges $g \in \binom{R}{2}$ for which there exists a
graph $G \in X_{0,5}(g)$, with $G \subseteq G(n,p) \setminus \binom{R}{2}$ and
$G \cap E \ne \emptyset$.  Let $E_2$ be the number of paths of length two in
$G(n,p) \cap \binom{R}{2}$ that complete some edge $g \in E_1$ to a triangle.
Start by exposing only the edges in $G(n,p) \setminus \binom{R}{2}$.
Conditioned on~(C3), we have that $|E_1| \le |E| \cdot n^{2/5 + 20 \epsilon_3}
= O(n^{4.5/5 + 20 \epsilon_3})$. Continue by exposing the edges in $G(n,p) \cap
\binom{R}{2}$. The number of possible paths of length two in $\binom{R}{2}$
that complete some edge $g \in E_1$ to a triangle is at most $|E_1| \cdot s  =
O(n^{7.6/5})$.  Therefore, assuming the maximum degree in $G(n,p) \cap
\binom{R}{2}$ is at most $n^{1.1/5}$, by Claim~\ref{claim:sec:6}, with
probability at least $1 - n^{-\omega(s)}$, $|E_2| = O(n^{7.6/5 - 3.99/5} +
n^{4.2/5})$. This completes the proof. 
\end{proof}

We continue with the following conditional lower bound on the probability of
$\D_i$.
\begin{lemma} \label{lemma:sec:4:2}
Assume that $\M_i$ is given so that $\A_i \wedge \B_i$ holds. Further assume
that the graph $\Trav_i$ has the following properties: first, the maximum
degree is at most $n^{3/5 + 10 \epsilon_3}$; second, the number of vertices
that are adjacent to any two fixed vertices is at most $n^{1/5 + 10
\epsilon_3}$; third, the number of vertices that are adjacent to any three
fixed vertices is at most $(\ln n)^{O(1)}$.  Then the probability of $\D_i$
(over the choice of $\BIGBite_{i+1}$) is at least $1 - n^{-\omega(1)}$.
\end{lemma}
\begin{proof}
We begin by noting that by Chernoff's bound,~(D1) occurs with probability at
least $1 - n^{-\omega(1)}$.

We continue by arguing that under the assumptions in the lemma,~(D2) occurs
with probability at least $1 - n^{-\omega(1)}$.  Let $1 \le j \le 5$ and $f \in
\NotTrav_i$.  
We assume that either $i \ge 1$, or else $j=5$, since otherwise trivially
$|X'_{i,j}(f)| = n^{(\epsilon_3 - 2/5)j} x_{i,j} = 0$ and we are done.
Let $\{G_l' : l \in L\}$ be the set $X_{i,j}(f)$.  Let $\{G_l : l \in L\}$ be
the family (potentially a multiset) for which it holds that $G_l = G_l' \cap
\NotTrav_i$ for every $l \in L$.  Consider the binomial random graph $G(n,p)$
with $p = n^{\epsilon_3 - 2/5}$, and let $W$ be as defined at the beginning of
Section~\ref{sec:3}.  Since $\A_i$ holds and $|L| = |X_{i,j}(f)|$, we have
$\expec(W) = n^{(\epsilon_3 - 2/5)j} |L|  \in n^{(\epsilon_3 - 2/5)j} x_{i,j}
(1 \pm 1000 \Gamma_i )$.  Observe that $|X'_{i,j}(f)|$ has the same
distribution as $W$ and so, by the union bound, it is enough to prove that the
probability that $W$ deviates from its expectation by more than $1000 \Gamma_i
n^{(\epsilon_3 - 2/5)j} x_{i,j} \ge \expec(W)^{0.9} = n^{\Omega(\epsilon_3)}$
is at most $n^{-\omega(1)}$.  We do so using Theorem~\ref{thm:vu}, with $\EE_0
= \expec(W), \EE_k = \exp((2j-k)\sqrt{\ln n})$ for $1 \le k \le j$, and
$\lambda = (\ln n)^2$.  We have several cases.
\begin{itemize}
\item Assume $j=5$. If $3 \le |G| \le 5$ then trivially $|L_G| \le 1$, while if
$1 \le |G| \le 2$ then trivially $|L_G| \le n$.  Therefore, $\expec_k(W) \le 1$
for all $1 \le k \le 5$.  Now note that $\expec_0(W) = \expec(W)$ and apply
Theorem~\ref{thm:vu} with the above parameters.
\item Assume $j=4$. If $3 \le |G| \le 4$ then trivially $|L_G| \le 1$.  If $|G|
= 2$ then since by assumption the maximum degree in the graph $\Trav_i$ is at
most $n^{3/5 + 10 \epsilon_3}$, we have $|L_G| = O(n^{3/5 + 10 \epsilon_3})$.
If $|G| = 1$ then trivially $|L_G| \le n$.  Therefore, $\expec_k(W) \le 1$ for
all $1 \le k \le 4$.  As before, note that $\expec_0(W) = \expec(W)$ and apply
Theorem~\ref{thm:vu} with the above parameters.
\item Assume $j=3$. If $|G| = 3$ then trivially $|L_G| \le 1$. If $|G| = 2$
then since by assumption the number of vertices that are adjacent in $\Trav_i$
to any two fixed vertices is at most $n^{1/5 + 10 \epsilon_3}$, we have $|L_G|
= O(n^{1/5 + 10 \epsilon_3})$. If $|G| = 1$ then since by assumption the
maximum degree in the graph $\Trav_i$ is at most $n^{3/5 +  10\epsilon_3}$, we
have $|L_G| = O(n^{3/5 + 10\epsilon_3})$. Therefore, $\expec_k(W) \le 1$ for
all $1 \le k \le 3$.  Apply Theorem~\ref{thm:vu} with the above parameters.
\item Assume $j=2$. If $|G| = 2$ then since by assumption the number of
vertices that are adjacent in $\Trav_i$ to any three fixed vertices is at most
$(\ln n)^{O(1)}$, we have $|L_G| \le (\ln n)^{O(1)}$. If $|G| = 1$ then since
by assumption the number of vertices that are adjacent in $\Trav_i$ to any two
fixed vertices is at most $n^{1/5 + 10 \epsilon_3}$, we have $|L_G| = O(n^{1/5
+ 10\epsilon_3})$. Therefore, $\expec_k(W) \le (\ln n)^{O(1)}$ for all $1 \le k
\le 2$.  Apply Theorem~\ref{thm:vu} with the above parameters.
\item Assume $j=1$. If $|G| = 1$ then since by assumption the number of
vertices that are adjacent in $\Trav_i$ to any three fixed vertices is at most
$(\ln n)^{O(1)}$, we have $|L_G| \le (\ln n)^{O(1)}$. Therefore, $\expec_1(W)
\le (\ln n)^{O(1)}$.  Apply Theorem~\ref{thm:vu} with the above parameters.
\end{itemize}

We end by arguing that under the assumptions in the lemma,~(D3) occurs with
probability at least $1 - n^{-\omega(1)}$.  Let $S \subseteq [n]$ be a set of
$s$ vertices, let $S_i \subseteq S$ be the set that is guaranteed to exist by
$\B_i$, let $1 \le k < j \le 3$, let $(R, T) \in \Pairs(S_i)$ and let $t =
|T|$, noting that $t = \Omega(s^3)$.
Note that by~(B2), $\Trav_i \cap \binom{R}{2}$ has maximum degree at most
$n^{1.1/5}$, and that by~(B3), $|Y_{i,j}(T)| \ge y_{i,j,t} (1 - 100\Gamma_i)$
(as $2 \le j \le 3$).
By the union bound, it is enough to show that $|Y'_{i,j,k}(T)| \ge
n^{(\epsilon_3 - 2/5)k} \binom{j}{k} y_{i,j,t} (1 - 100 \Gamma_i - \Gamma_i
\gamma_i)$ occurs with probability at least $1 - n^{-\omega(s)}$.  
Let $\{G_l : l \in L\}$ be the multiset $\biguplus_{(G_1, G_2, G_3)} G_2$,
where the multiset union ranges over all $(G_1, G_2, G_3) \in Y_{i,j,k}(T)$.
Consider the binomial random graph $G(n,p)$ with $p = n^{\epsilon_3 - 2/5}$,
and let $W$ be as defined at the beginning of Section~\ref{sec:3}.  By
assumption, $|L| = \binom{j}{k} |Y_{i,j}(T)| \ge \binom{j}{k} y_{i,j,t} (1 -
100 \Gamma_i)$ and so $\expec(W) \ge n^{(\epsilon_3 - 2/5)k} \binom{j}{k}
y_{i,j,t} (1 - 100 \Gamma_i)$.
Since $|Y'_{i,j,k}(T)|$ has the same distribution as $W$, it remains for us to
show that the probability that $W < \expec(W) - \lambda$ for $\lambda =
\Gamma_i \gamma_i n^{(\epsilon_3 - 2/5)k} \binom{j}{k} y_{i,j,t} $, is at most
$n^{-\omega(s)}$.  We do so using Theorem~\ref{eq:janson}, and for that it is
enough to show that $\lambda^2 / (\expec(W) + \Delta) = \omega(s \ln n)$, where
$\Delta$ is as defined in Section~\ref{sec:3}. We have three cases.
\begin{itemize}
\item Assume that $k=2$ and $j=3$. Clearly, $\Delta \le 12 s^4 n^{3(\epsilon_3
- 2/5)} \le n^{6.1/5}$. Furthermore, $\expec(W) \le 3 s^3 n^{2(\epsilon_3 -
2/5)}\le n^{5.1/5}$. Also, $\lambda \ge n^{5/5}$ by Lemma~\ref{fact:f1}.  Thus,
$\lambda^2 / (\expec(W) + \Delta) = \omega(s \ln n)$.
\item Assume that $k=1$ and $j=3$. Clearly, $\Delta \le 3 s^4 n^{\epsilon_3 -
2/5} \le n^{10.1/5}$.  Furthermore, $\expec(W) \le 3 s^3 n^{\epsilon_3 - 2/5}
\le n^{7.1/5}$. Also, $\lambda \ge n^{7/5}$ by Lemma~\ref{fact:f1}.  Thus,
$\lambda^2 / (\expec(W) + \Delta) = \omega(s \ln n)$.
\item Assume that $k=1$ and $j=2$. Recall that $|L| \ge \binom{j}{k} y_{i,j,t}
(1 - 100 \Gamma_i)$.  It is safe to assume that $|L| \le \binom{j}{k} y_{i,j,t}
$ (since otherwise we can remove some of the members of $L$ so that this
assumption does hold; such an alteration will not affect the proof). By this
upper bound on $|L|$, and using the fact that the maximum degree in $\Trav_i
\cap \binom{R}{2}$ is at most $n^{1.1/5}$, one can verify using
Lemma~\ref{fact:f1} that $\Delta \le 2 |L| n^{1.1/5} n^{\epsilon_3 - 2/5} \le
n^{6.2/5}$.  Furthermore, by this upper bound on $|L|$ and by
Lemma~\ref{fact:f1}, $\expec(W) = |L| n^{\epsilon_3 - 2/5} \le n^{5.1/5}$.
Also, $\lambda \ge n^{5/5}$ by Lemma~\ref{fact:f1}.  Thus, $\lambda^2 /
(\expec(W) + \Delta) = \omega(s \ln n)$.
\end{itemize}
\end{proof}

\section{Proof of Lemma~\ref{lemma0}} \label{sec:7}
%
In this section we prove our main lemma.  We start with the following lemma.
\begin{lemma} \label{lemma:sec:7:2}
For $0 \le i < I$,
\begin{eqnarray*}
\prob(\A_i \wedge \B_i) \ge 1 - i n^{-0.1} \implies \prob(\A_i \wedge \B_i
\wedge \C_i \wedge \D_i) \ge 1 - i n^{-0.1} - n^{-\omega(1)}.
\end{eqnarray*}
\end{lemma}
\begin{proof}
Let $0 \le i < I$.  Assume that $\prob(\A_i \wedge \B_i) \ge 1 - i n^{-0.1}$.
Standard arguments show that with probability at least $1 - n^{-\omega(1)}$,
the graph $\Trav_i$ satisfies the three properties that are assumed to hold in
Lemma~\ref{lemma:sec:4:2}.  Therefore, by Lemma~\ref{lemma:sec:4:2} we have
$\prob(\A_i \wedge \B_i \wedge \D_i) \ge 1 - i n^{-0.1} - n^{-\omega(1)}$. The
lemma now follows since by Lemma~\ref{lemma:sec:4:1} we have $\prob(\C_i) \ge 1
- n^{-\omega(1)}$.
\end{proof}

The proof of Lemma~\ref{lemma0} is by induction on $i$. Trivially, $\prob(\A_0
\wedge \B_0) = 1$. This, together with Lemma~\ref{lemma:sec:7:2}, implies the
validity of the lemma for the case $i=0$.  Assume the lemma holds for  $0 \le i
< I-1$.  We prove that $\prob(\A_{i+1} \wedge \B_{i+1} \wedge \C_{i+1} \wedge
\D_{i+1}) \ge 1 - (i+1) n^{-0.1} - n^{-\omega(1)}$.  To do that, by
Lemma~\ref{lemma:sec:7:2} it suffices to prove that $\prob(\A_{i+1} \wedge
\B_{i+1}) \ge 1 - (i+1) n^{-0.1}$.  Thus, by the induction hypothesis,
it is enough to prove that
\begin{eqnarray}
\label{eq:sec:7:1} \prob(\A_{i+1} \given \A_i \wedge \B_i \wedge \C_i \wedge
\D_i) &\ge& 1 - n^{-0.11} \text{\quad and }\\ 
\label{eq:sec:7:2} \prob(\B_{i+1} \given \A_i \wedge \B_i \wedge \C_i \wedge
\D_i) &\ge& 1 - n^{-0.11}.
\end{eqnarray}

The proof of~(\ref{eq:sec:7:1})~and~(\ref{eq:sec:7:2}) is given in the next two
subsections.  In our arguments below we make use of the notion of an outcome of
an edge in $\BIGBite_{i+1}$: an outcome of such an edge is either the event
that the edge in not in $\Bite_{i+1}$, or otherwise it is the birthtime of the
edge. 
Changing the outcome of an edge that is not in $\Bite_{i+1}$ means adding that
edge to $\Bite_{i+1}$ and giving it an arbitrary birthtime.  Changing the
outcome of an edge in $\Bite_{i+1}$ means either taking that edge out of
$\Bite_{i+1}$, or changing its birthtime arbitrarily.

We will also need the following definitions and observations.  Given
$\BIGBite_{i+1}$, say that an edge $f \in \BIGBite_{i+1}$ has the potential of
being a label in a tree $\TT_c(g)$ or in a tree $\TT_c(g,S)$, where $g \in
\NotTrav_i$ and $S \subseteq [n]$, if there exists a choice of $\BigBite_{i+1}
\subseteq \BIGBite_{i+1}$ so that given that particular choice of
$\BigBite_{i+1}$, indeed $f$ is a label in $\TT_c(g)$ or $\TT_c(g,S)$,
respectively. (Recall that $\BigBite_{i+1}$ completely determines $\TT_c(g)$
and $\TT_c(g,S)$, and that those labels of $\TT_c(g)$ and $\TT_c(g,S)$ that are
edges, are all in $\BigBite_{i+1}$, except maybe for the label of the root.)
The motivation behind these definitions is that, for example, whenever $f \in
\BIGBite_{i+1}$ has the potential of being a label in a tree $\TT_c(g)$, then
we know that the outcome of $f$ could affect the occurrence of $\S_c(g)$;
otherwise, the outcome of $f$ would not affect the occurrence of $\S_c(g)$. 
For an edge $f \in \NotTrav_i$, let $\Labels_c(f)$ be the set of edges in
$\BIGBite_{i+1}$ which have the potential of being labels in $\TT_c(f)$.
For an edge $f \in \BIGBite_{i+1}$, let $\Roots_c(f)$ be set of edges $g \in
\NotTrav_i$ such that $f$ has the potential of being a label in $\TT_c(g)$.
Assuming $\A_i \wedge \B_i \wedge \C_i \wedge \D_i$ and $c \in \epsilon_3^2
\epsilon_2^{-1} \pm 1$, it follows from~(D2),~(C5) and from our choice of
$\epsilon_2$ that
\begin{eqnarray} 
\label{eq:labels}
f \in \NotTrav_i &\implies& |\Labels_c(f)| \le 1 + 5^c \cdot n^{10\epsilon_3 c}
\le n^{0.01}, \\
\label{eq:roots1}
f \in \BIGBite_{i+1} &\implies& |\Roots_c(f) \cap \BIGBite_{i+1}| \le 1 + 5^c
\cdot n^{10\epsilon_3 c} \le n^{0.01}  \text{\quad  and } \\
\label{eq:roots2}
f \in \BIGBite_{i+1} &\implies& |\Roots_c(f)| \le 1 + 5^c \cdot n^{10\epsilon_3
c} \cdot n^{2/5 + 10\epsilon_3} \le n^{2.1/5}.
\end{eqnarray}

\subsection{Proof of~(\ref{eq:sec:7:1})} \label{sec:7:1}
%
Let $\E_4$ and $\E_5$ be, respectively, the events 
\begin{eqnarray*}
|\M_{i+1}| &\in& 0.5 n^{8/5} \Phi((i+1) n^{-\epsilon_1}) (1 \pm 100
\Gamma_{i+1}) \text{\quad and } \\
|O_{i+1}| &\in& 0.5 n^2 \phi((i+1) n^{-\epsilon_1}) (1 \pm 100 \Gamma_{i+1}).
\end{eqnarray*}
\begin{lemma} \label{lemma:sec:7:1}
Assume that $\M_i$ and $\BIGBite_{i+1}$ are given so that $\A_i \wedge \B_i
\wedge \C_i \wedge \D_i$ holds.  Then the probability of $\E_4 \wedge \E_5$
(over the choice of $\BigBite_{i+1}$, $\Bite_{i+1}$ and the choice of the
birthtimes of the edges in $\Bite_{i+1}$) is at least $1 - n^{-\omega(1)}$.
\end{lemma}
\begin{proof}
Fix an integer $c \in \epsilon_3^2\epsilon_2^{-1} \pm 1$. Let $\sum_f$ range
over all $f \in O_i$.  Let 
\begin{eqnarray*}
W_1 \deq \sum_f \ONE[\S_c(f) \wedge f \in \Bite_{i+1}] \text{\quad and \quad } 
W_2 \deq \sum_f \ONE[\S_c(f) \wedge f \notin \Bite_{i+1}].
\end{eqnarray*}
By Lemma~\ref{lemma:conn}, we have that if $c$ is odd then $|\M_{i+1}
\setminus \M_i| \ge W_1$ and $|O_{i+1}| \ge W_2$, while if $c$ is even then
$|\M_{i+1} \setminus \M_i| \le W_1$ and $|O_{i+1}| \le W_2$.  Now,~(A1) and the
fact that $\Gamma_i \le \Gamma_{i+1}$ imply that $|\M_i| \in 0.5 n^{8/5} \Phi(i
n^{-\epsilon_1}) (1 \pm 100 \Gamma_{i+1})$. So it suffices to prove that $W_1
\in 0.5 n^{8/5} (\Phi((i+1) n^{-\epsilon_1}) - \Phi(i n^{-\epsilon_1})) (1 \pm
100 \Gamma_{i+1})$ and $W_2 \in 0.5 n^2 \phi((i+1)n^{-\epsilon_1}) (1 \pm 100
\Gamma_{i+1})$, each occurs with probability at least $1 - n^{-\omega(1)}$.

Since by~(A2) the number of edges over which $\sum_f$ ranges is in $0.5 n^2
\phi(i n^{-\epsilon_1}) (1 \pm 100 \Gamma_i)$, since by~(D1) the number of
edges in $\BIGBite_{i+1}$ over which $\sum_f$ ranges is in $0.5 n^{8/5 +
\epsilon_3} \phi(i n^{-\epsilon_1}) (1 \pm (100 \Gamma_i + \Gamma_i
\gamma_i))$, and since $\C_i \wedge \D_i$ holds, we may apply
Lemma~\ref{lemma:survive:process} to get
\begin{eqnarray*} 
\expec(W_1) &\in& 0.5 n^{8/5} (\Phi((i+1)n^{-\epsilon_1}) - \Phi(i
n^{-\epsilon_1}) ) (1 \pm ( 100 \Gamma_i + 99 \Gamma_i \gamma_i )) \text{\quad and }
\\
\expec(W_2) &\in& 0.5 n^{2} \phi((i+1)n^{-\epsilon_1}) (1 \pm ( 100 \Gamma_i +
99 \Gamma_i \gamma_i )).  
\end{eqnarray*}
Observe that the above estimate on $\expec(W_2)$, together with
Lemma~\ref{fact:f1}, implies that $\expec(W_2) \ge n^{9.9/5}$.  Also, the above
estimate on $\expec(W_1)$, together with Lemma~\ref{fact:f1}, and the fact that
$\Phi((i+1) n^{-\epsilon_1}) - \Phi(i n^{-\epsilon_1}) \ge n^{- \Theta(
\epsilon_1 )}$ (indeed, recall the proof of Lemma~\ref{lemma:survive:process},
where we have argued indirectly that $(\Phi((i+1) n^{-\epsilon_1}) - \Phi(i
n^{-\epsilon_1})) / ( n^{-\epsilon_1} \phi(i n^{-\epsilon_1}) ) = \Omega(1)$),
implies that $\expec(W_1) \ge n^{7.9/5}$.  Therefore, by Lemma~\ref{fact:f1},
it suffices to show that the probability that $W_1$ and $W_2$ deviate from
their expectation by more than $n^{7/5}$ is at most $n^{-\omega(1)}$.

Note that $W_1$ and $W_2$ each depends only on the outcomes of edges in
$\BIGBite_{i+1}$, which by~(C1) contains at most $n^{8.1/5}$ edges.
Furthermore, by~(\ref{eq:roots2}), every edge in $\BIGBite_{i+1}$ has the
potential of being a label in at most $n^{2.1/5}$ trees $\TT_c(f)$ with $f \in
O_i$.  This implies that changing the outcome of a single edge in
$\BIGBite_{i+1}$ can change $W_1$ and $W_2$ each by at most an additive factor
of $n^{2.1/5}$.  Therefore, by McDiarmid's inequality we can conclude that the
probability that $W_1$ or $W_2$ each deviates from its expectation by more than
$n^{7/5}$ is at most $n^{-\omega(1)}$.
\end{proof}
From Lemma~\ref{lemma:sec:7:1} it follows that
\begin{eqnarray*}
\prob(\E_4 \wedge \E_5 \given \A_i \wedge \B_i \wedge \C_i \wedge \D_i) \ge 1 -
n^{-\omega(1)}.
\end{eqnarray*}
Let $\E_6$ be the following event: letting $m_{i+1} \deq |\M_{i+1}|$, for all
$1 \le j \le 5$ and all $f \in \NotTrav_i$,
\begin{eqnarray*} 
|O_{i+1}| &\in& 0.5 n^2 \exp(-16 (m_{i+1} n^{-8/5})^5) (1 \pm n^{-\epsilon_3})
\quad \text{ and } \\
|X_{i+1,j}(f)| &\in& n^{2j/5} 2^{4-j} \tbinom{5}{j} ( m_{i+1} n^{-8/5} )^{5-j}
\exp(-16 j (m_{i+1} n^{-8/5})^5) (1 \pm n^{-\epsilon_3}).
\end{eqnarray*}
A result of Bohman~\cite[Theorem~13]{MR2522430}
implies that $\prob(\E_6) \ge 1 - n^{-1/6}$, and so by
the induction hypothesis,
\begin{eqnarray*}
\prob(\E_6 \given \A_i \wedge \B_i \wedge \C_i \wedge \D_i) \ge
\frac{\prob(\A_i \wedge \B_i \wedge \C_i \wedge \D_i) - n^{-1/6}}{\prob(\A_i
\wedge \B_i \wedge \C_i \wedge \D_i)} \ge 1 - 2 n^{-1/6}. 
\end{eqnarray*}
Since our goal is to prove~(\ref{eq:sec:7:1}), it  remains to argue that $\E_4
\wedge \E_5 \wedge \E_6$ imply the bounds on $|X_{i+1,j}(f)|$ that are asserted
by $\A_{i+1}$, for all $1 \le j \le 5$ and all $f \in \NotTrav_i$.  Indeed,
note that $\E_4 \wedge \E_5 \wedge \E_6$ implies 
\begin{eqnarray*}
m_{i+1} &\in& 0.5 n^{8/5} \Phi((i+1) n^{-\epsilon_1}) (1 \pm 100 \Gamma_{i+1})
\text{\quad and } \\
\exp(-16 (m_{i+1} n^{-8/5})^5) &\in& \phi((i+1) n^{-\epsilon_1}) (1 \pm 101
\Gamma_{i+1}).
\end{eqnarray*}
This in turn implies that for all $1 \le j \le 5$,
\begin{eqnarray*}
2^{4-j} (m_{i+1} n^{-8/5})^{5 - j} \exp(-16 j (m_{i+1} n^{-8/5})^5) (1 \pm
n^{-\epsilon_3}) 
&\subseteq& \\
0.5 (\Phi((i+1) n^{-\epsilon_1}))^{5-j} \phi((i+1) n^{-\epsilon_1})^j (1 \pm
999 \Gamma_{i+1}).
\end{eqnarray*}
This, together with the fact that $0.5 n^{2j/5} \in \binom{n}{2}
\big(\frac{1}{n^{2/5}}\big)^{5-j} (1 \pm o(\Gamma_{i+1}))$, completes the
proof.

\subsection{Proof of~(\ref{eq:sec:7:2})} \label{sec:7:2}
%
Fix for the rest of the section a set $S \subseteq [n]$ of $s$ vertices.
Further, assume that we are given $\M_i$ and $\BIGBite_{i+1}$ so that $\A_i
\wedge \B_i \wedge \C_i \wedge \D_i$ holds.  Under this assumption, we prove
that with probability at least $1 - n^{-\omega(s)}$ (where the probability is
over the choice of $\BigBite_{i+1}$, $\Bite_{i+1}$ and the choice of the
birthtimes of the edges in $\Bite_{i+1}$), there exists a set $S_{i+1}
\subseteq S$, which satisfies the three properties that are asserted to hold by
$\B_{i+1}$.  A union bound argument will then give us~(\ref{eq:sec:7:2}).

We start by defining the set $S_{i+1}$.  Let $S_i \subseteq S$ be the set that
is guaranteed to exist by $\B_i$. Let $S_{i+1}$ be the set of all vertices in
$S_i$ whose degree in $(\Trav_i \cup \BIGBite_{i+1}) \cap \binom{S_i}{2}$ is at
most $n^{1.1/5}$, so that
\begin{eqnarray} \label{eq:maxdeg}
\text{the maximum degree in $(\Trav_i \cup \BIGBite_{i+1}) \cap
\tbinom{S_{i+1}}{2}$ is at most $n^{1.1/5}$}.
\end{eqnarray}
We claim that with probability $1$, $S_{i+1}$ satisfies the first two
properties that are asserted to hold by $\B_{i+1}$.  Indeed, by assumption,
$S_i$ has size at least $s (1 - i n^{-0.01}) = \omega(n^{3/5})$.  This,
together with~(C2), implies that the average degree in $(\Trav_i \cup
\BIGBite_{i+1}) \cap \binom{S_i}{2}$ is at most $n^{1/5 + 10 \epsilon_3}$.
Hence, the number of vertices in $S_i$ whose degree in $(\Trav_i \cup
\BIGBite_{i+1}) \cap \binom{S_i}{2}$ is more than $n^{1.1/5}$ is at most $|S_i|
n^{-0.01} \le sn^{-0.01}$. It follows that $S_{i+1}$ has size at least $|S_i| -
sn^{-0.01} \ge s (1 - (i+1) n^{-0.01})$.  In addition to that, from the fact
that $\Trav_{i+1} \subseteq \Trav_i \cup \BIGBite_{i+1}$ and
from~(\ref{eq:maxdeg}) it follows that $\Trav_{i+1} \cap \binom{S_{i+1}}{2}$
has maximum degree at most $n^{1.1/5}$. 

It is left for us to show that with the desired probability, $S_{i+1}$
satisfies the third property that is asserted to hold by $\B_{i+1}$.  Fix for
the rest of the section a pair $(R, T) \in \Pairs(S_{i+1})$ and let $t = |T|$,
noting that $t = \Omega(s^3)$.
By the union bound, it remains to argue that for every $1 \le j \le 3$, with
probability at least $1 - n^{-\omega(s)}$,  $R$ and $T$ satisfy~$\B_{i+1}$(B3).
That is, it remains to prove that for every $1 \le j \le 3$,
\begin{eqnarray}  \label{eq:sec:7:3}
\prob( |Y_{i+1,j}(T)| \ge y_{i+1,j,t} (1 - 100 \Gamma_{i+1}) - 0.5 j (3-j)
(2-j) |Z_{i+1}(R,T)| ) \ge 1 - n^{-\omega(s)},
\end{eqnarray}
where we stress again that the probability is over the choice of
$\BigBite_{i+1}$, $\Bite_{i+1}$, and the choice of the birthtimes of the edges
in $\Bite_{i+1}$.  

In order to prove~(\ref{eq:sec:7:3}), we assume that the inequalities
\begin{equation} \label{eq:7:2:1}
|Y_{i,j}(T)| \le y_{i,j,t} \text{\quad and \quad } 
|Y'_{i,j,k}(T)| \le n^{(\epsilon_3 - 2/5)k} \tbinom{j}{k} y_{i,j,t}
\end{equation}
hold for every possible choice of $j$ and $k$.  This is a safe assumption since
by~(B3)~and~(D3) we can always remove, if needed, elements from these sets so
that this assumption holds. Such an alteration will not affect the proof.

\subsubsection{Proof of~(\ref{eq:sec:7:3}) (case $j=3$)}
%
Let $c \in \epsilon_3^2 \epsilon_2^{-1} \pm 1$ be an odd integer and let
\begin{eqnarray*}
W_3 \deq \sum_{G \in Y_{i,3}(T)} \ONE[\I_c(G) \wedge \S_c(G) \wedge |G \cap
\Bite_{i+1}| = 0].
\end{eqnarray*}
By Lemma~\ref{lemma:conn}, we have $|Y_{i+1,3}(T)| \ge W_3$.  This, with the
next two claims, gives~(\ref{eq:sec:7:3}) for $j=3$.

\begin{claim} 
$\expec(W_3) \ge y_{i+1,3,t} (1 - 100 \Gamma_i - 99 \Gamma_i \gamma_i)$.
\end{claim}
\begin{proof}
Consider a triangle $G \in Y_{i,3}(T)$. Since an edge in $\BIGBite_{i+1}$ is an
edge in $\Bite_{i+1}$ with probability $n^{-\epsilon_1 - \epsilon_3} / (1 - i
n^{-\epsilon_1 - \epsilon_2})$, using Lemma~\ref{fact:f1}, we find that
$\prob(|G \cap \Bite_{i+1}| = 0) \ge 1 - o(\Gamma_i \gamma_i)$. Hence, by
Lemma~\ref{lemma:survive:process}, 
\begin{eqnarray*}
\prob(\I_c(G) \wedge S_c(G) \wedge |G \cap \Bite_{i+1}| = 0) \ge
\bigg(\frac{\phi((i+1)n^{-\epsilon_1})} {\phi(i n^{-\epsilon_1})}\bigg)^3 (1 -
91 \Gamma_i \gamma_i).
\end{eqnarray*}
Also, by~(B3), $|Y_{i,3}(T)| \ge y_{i,3,t} (1 - 100\Gamma_i)$.  The claim now
follows using linearity of expectation.
\end{proof}
\begin{claim}
The probability (over the choice of $\BigBite_{i+1}$, $\Bite_{i+1}$, and the
choice of the birthtimes of the edges in $\Bite_{i+1}$) that $W_3$ deviates
from its expectation by more than $\Gamma_i \gamma_i y_{i+1,3,t}$ is at most
$n^{-\omega(s)}$.
\end{claim}
\begin{proof}
For an edge $f \in \BIGBite_{i+1}$, let $\Triangles(f)$ be the set which
contains every triangle $G \in Y_{i,3}(T)$ for which it holds that $f$ has the
potential of being a label in a tree in the forest $\{\TT_c(g) : g \in G\}$.
Observe that changing the outcome of an edge $f \in \BIGBite_{i+1}$ can change
$W_3$ by at most an additive factor of $|\Triangles(f)|$.

By~(\ref{eq:roots2}) we have that for every $f \in \BIGBite_{i+1}$,
$|\Triangles(f)| \le |\Roots_c(f)| \cdot s \le n^{5.2/5}$.
By~(\ref{eq:labels}) we have that for every triangle $G \in Y_{i,3}(T)$, there
are at most $3 n^{0.01}$ edges $f \in \BIGBite_{i+1}$ such that $G \in
\Triangles(f)$, and so
\begin{eqnarray*}
\sum_{f \in \BIGBite_{i+1}} |\Triangles(f)| \le |Y_{i,3}(T)| \cdot 3 n^{0.01}
\le n^{9.1/5}, 
\end{eqnarray*}
where the second inequality follows since trivially $|Y_{i,3}(T)| \le s^3$.
Therefore,
\begin{eqnarray*}
\sum_{f \in \BIGBite_{i+1}} |\Triangles(f)|^2 \le n^{5.2/5} \cdot \sum_{f \in
\BIGBite_{i+1}} |\Triangles(f)| \le n^{14.3/5}.
\end{eqnarray*}
It now follows from McDiarmid's inequality that the probability that $W_3$ deviates
from its expectation by more than $\Gamma_i \gamma_i y_{i+1,3,t} \ge n^{8.9/5}$
is at most $n^{-\omega(s)}$.
\end{proof}

\subsubsection{Proof of~(\ref{eq:sec:7:3}) (case $j=2$)}
%
Let $c \in \epsilon_3^2 \epsilon_2^{-1} \pm 1$ be an odd integer.  Let $\sum_G$
range over all triangles $G \in Y_{i,2}(T)$ and let $\sum_{(G_1, G_2, G_3)}$
range over all triples $(G_1, G_2, G_3) \in Y'_{i,3,1}(T)$.  Let
\begin{eqnarray*}
W_4 &\deq& \sum_{G} \ONE[\I_c(G \cap \NotTrav_i) \wedge \S_c(G \cap \NotTrav_i)
\wedge |G \cap \Bite_{i+1}|=0] + \\
&& \sum_{(G_1, G_2, G_3)}
\ONE[\I_c(G_2 \cup G_3) \wedge \S_c(G_2 \cup G_3) \wedge G_2 \subseteq
\Bite_{i+1} \wedge |G_3 \cap \Bite_{i+1}|=0 ].
\end{eqnarray*}
By Lemma~\ref{lemma:conn}, we have $|Y_{i+1,2}(T)| \ge W_4$.  This, with the
next two claims, gives~(\ref{eq:sec:7:3}) for $j=2$.

\begin{claim} \label{claim:w4}
$\expec(W_4) \ge y_{i+1,2,t} (1 - 100\Gamma_i - 99 \Gamma_i \gamma_i)$.
\end{claim}
\begin{proof}
Consider a triangle $G \in Y_{i,2}(T)$ and a triple $(G_1, G_2, G_3) \in
Y'_{i,3,1}(T)$.  Since an edge in $\BIGBite_{i+1}$ is an edge in $\Bite_{i+1}$
with probability $n^{-\epsilon_1 - \epsilon_3} / (1 - i n^{-\epsilon_1 -
\epsilon_2})$, using Lemma~\ref{fact:f1}, we find that 
\begin{eqnarray*}
\prob(|G \cap \Bite_{i+1}| = 0) &\ge& 1 - o(\Gamma_i \gamma_i) \text{ \quad and
} \\
\prob(G_2 \subseteq \Bite_{i+1} \wedge |G_3 \cap \Bite_{i+1}| = 0) &\ge&
n^{-\epsilon_1 - \epsilon_3}(1 - o(\Gamma_i \gamma_i)).
\end{eqnarray*}
Hence, by Lemma~\ref{lemma:survive:process}, 
\begin{eqnarray*}
\prob(\I_c(G \cap \NotTrav_i) \wedge \S_c(G \cap \NotTrav_i) \wedge |G \cap
\Bite_{i+1}| = 0) &\ge& \\
\bigg(\frac{\phi((i+1)n^{-\epsilon_1})} {\phi(i n^{-\epsilon_1})}\bigg)^2 (1 -
91 \Gamma_i \gamma_i)
\end{eqnarray*}
and 
\begin{eqnarray*}
\prob(\I_c(G_2 \cup G_3) \wedge \S_c(G_2 \cup G_3) \wedge G_2 \subseteq
\Bite_{i+1} \wedge |G_3 \cap \Bite_{i+1}| = 0) &\ge&  \\
n^{-\epsilon_1 - \epsilon_3} \bigg( \frac{\Phi((i+1)n^{-\epsilon_1}) -
\Phi(in^{-\epsilon_1})}{n^{-\epsilon_1} \phi(in^{-\epsilon_1})} \bigg)
\bigg(\frac{\phi((i+1)n^{-\epsilon_1})} {\phi(i n^{-\epsilon_1})}\bigg)^2 (1 -
91 \Gamma_i \gamma_i).
\end{eqnarray*}
Also,  by~(B3), $|Y_{i,2}(T)| \ge y_{i,2,t} (1 - 100\Gamma_i)$, and by~(D3),
$|Y'_{i,3,1}(T)| \ge 3 n^{\epsilon_3 - 2/5} y_{i,3,t} (1 - 100 \Gamma_i -
\Gamma_i \gamma_i )$.  The claim now follows using linearity of expectation.
\end{proof}

\begin{claim}
The probability (over the choice of $\BigBite_{i+1}$, $\Bite_{i+1}$, and the
choice of the birthtimes of the edges in $\Bite_{i+1}$) that $W_4$ deviates
from its expectation by more than $\Gamma_i \gamma_i y_{i+1,2,t}$ is at most
$n^{-\omega(s)}$.
\end{claim}
\begin{proof}
For an edge $f \in \BIGBite_{i+1}$, let $\Triangles(f)$ be the set which
contains every triangle $G \in Y_{i,2}(T)$ for which it holds that $f$ has the
potential of being a label in a tree in the forest $\{\TT_c(g) : g \in G \cap
\NotTrav_i\}$, and every triple $(G_1, G_2, G_3) \in Y'_{i,3,1}(T)$ for which
it holds that $f$ has the potential of being a label in a tree in the forest
$\{\TT_c(g) : g \in G_2 \cup G_3\}$.
Observe that changing the outcome of an edge $f \in \BIGBite_{i+1}$ can change
$W_4$ by at most an additive factor of $|\Triangles(f)|$.

Note that every triangle $G \in Y_{i,2}(T)$ has at least one edge in $\Trav_i
\cup \BIGBite_{i+1}$, and that every triple $(G_1, G_2, G_3) \in Y'_{i,3,1}(T)$
is such that the triangle $G_1 \cup G_2 \cup G_3$ has at least one edge in
$\Trav_i \cup \BIGBite_{i+1}$. This fact, together
with~(\ref{eq:roots1}),~(\ref{eq:roots2}),~(\ref{eq:maxdeg}) and the fact that
for every triangle $G$ there are at most $3$ possible triples $(G_1, G_2, G_3)
\in Y'_{i,3,1}(T)$ such that $G = G_1 \cup G_2 \cup G_3$, gives us that for
every $f \in \BIGBite_{i+1}$, 
\begin{eqnarray*}
|\Triangles(f)| \le 3 \cdot |\Roots_c(f)| \cdot 2n^{1.1/5} + 3 \cdot
|\Roots_c(f) \cap \BIGBite_{i+1}| \cdot s \le 10 n^{3.2/5}.
\end{eqnarray*}
Moreover, by~(\ref{eq:labels}) we have that for every triangle $G \in
Y_{i,2}(T)$, there are at most $2n^{0.01}$ edges $f \in \BIGBite_{i+1}$ such
that $G \in \Triangles(f)$, and likewise for every triple $(G_1, G_2, G_3) \in
Y'_{i,3,1}(T)$, there are at most $3 n^{0.01}$ edges $f \in \BIGBite_{i+1}$
such that $(G_1, G_2, G_3) \in \Triangles(f)$, and so
\begin{eqnarray*}
\sum_{f \in \BIGBite_{i+1}} |\Triangles(f)| \le (|Y_{i,2}(T)| +
|Y'_{i,3,1}(T)|) \cdot 3n^{0.01} \le n^{7.1/5},
\end{eqnarray*}
where the second inequality follows from~(\ref{eq:7:2:1}).  Therefore, 
\begin{eqnarray*}
\sum_{f \in \BIGBite_{i+1}} |\Triangles(f)|^2 \le 10 n^{3.2/5} \cdot \sum_{f
\in \BIGBite_{i+1}} |\Triangles(f)| \le 10 n^{10.3/5}.
\end{eqnarray*}
It now follows from McDiarmid's inequality that the probability that $W_4$ deviates
from its expectation by more than $\Gamma_i \gamma_i y_{i+1,2,t} \ge n^{6.9/5}$
is at most $n^{-\omega(s)}$.
\end{proof}

\subsubsection{Proof of~(\ref{eq:sec:7:3}) (case $j=1$)}
%
We begin with the following claim.
\begin{claim} \label{claim:7:2:1}
Let $Y_1(T) \deq Y_{i,1}(T) \cup Y'_{i,2,1}(T) \cup Y'_{i,3,2}(T)$ be a set of
triangles and triples.  There exists a set $Y(T) \subseteq Y_1(T)$, of size at
least $|Y_1(T)| - O(n^{4.2/5})$, such that
for every edge $g_1$ there is at most one triangle $\{g_1, g_2, g_3\}$ with
$\{g_2, g_3\} \subseteq \Trav_i \cup \BIGBite_{i+1}$, such that either $\{g_1,
g_2, g_3\} \in Y(T)$, or $\{g_1, g_2, g_3\} = G_1 \cup G_2 \cup G_3$ for some
triple $(G_1, G_2, G_3) \in Y(T)$.
\end{claim}
\begin{proof}
By~(C6), there is a set $E_0$ of at most $n^{3/5 + 10 \epsilon_3}$ edges, the
removal of which from $\Trav_i \cup \BIGBite_{i+1}$ leaves at most $n^{4/5 + 10
\epsilon_3}$ $4$-cycles in $(\Trav_i \cup \BIGBite_{i+1}) \cap \binom{R}{2}$.
Obtain $Y_2(T)$ from $Y_1(T)$ by removing from $Y_1(T)$ every triangle $G$ for
which it holds that $G \cap E_0 \ne \emptyset$, and every triple $(G_1, G_2,
G_3)$ for which it holds that $(G_1 \cup G_2 \cup G_3) \cap E_0 \ne \emptyset$.
Note that every triangle in $Y_1(T)$ has at least two edges in $\Trav_i \cup
\BIGBite_{i+1}$, and that every triple $(G_1, G_2, G_3)$ in $Y_1(T)$ is such
that the triangle $G_1 \cup G_2 \cup G_3$ has at least two edges in $\Trav_i
\cup \BIGBite_{i+1}$.  Hence, using~(\ref{eq:maxdeg}), we find that every edge
in $E_0$ belongs to at most $2 n^{1.1/5}$ triangles in $Y_1(T)$, and to at most
$2 n^{1.1/5}$ triangles $G_1 \cup G_2 \cup G_3$ such that $(G_1, G_2, G_3) \in
Y_1(T)$. Since for every triangle $G$ there are at most $3$ triples $(G_1, G_2,
G_3) \in Y_1(T)$ such that $G = G_1 \cup G_2 \cup G_3$, we get that $|Y_2(T)| =
|Y_1(T)| - |E_0| \cdot O(n^{1.1/5})$, and so $|Y_2(T)| = |Y_1(T)| -
O(n^{4.2/5})$. 

Say that a triangle $G \in Y_2(T)$ is bad if it has two edges in $\Trav_i \cup
\BIGBite_{i+1}$ that belong to some $4$-cycle in $(\Trav_i \cup \BIGBite_{i+1})
\cap \binom{R}{2}$. Say that a triple $(G_1, G_2, G_3) \in Y_2(T)$ is bad if
the triangle $G_1 \cup G_2 \cup G_3$  has two edges in $\Trav_i \cup
\BIGBite_{i+1}$ that belong to some $4$-cycle in $(\Trav_i \cup \BIGBite_{i+1})
\cap \binom{R}{2}$. 
Obtain $Y(T)$ from $Y_2(T)$ by removing from $Y_2(T)$ every bad triangle and
every bad triple. Observe that $Y(T)$ satisfies the property that is asserted
to hold by the claim, and so it remains for us to show that $|Y(T)| = |Y_1(T)|
- O(n^{4.2/5})$.  To show this, first recall that if we remove from $\Trav_i
\cup \BIGBite_{i+1}$ the edges in $E_0$, the number of $4$-cycles remaining in
$(\Trav_i \cup \BIGBite_{i+1}) \cap \binom{R}{2}$ is at most $n^{4/5 +
10\epsilon_3} \le n^{4.1/5}$.  Then, note that for every $4$-cycle there are
exactly $4$ triangles that share two of their edges with that cycle.  It
follows that the number of bad triangles and bad triples in $Y_2(T)$ is at most
$O(n^{4.1/5})$. Thus, $|Y(T)| = |Y_2(T)| - O(n^{4.1/5})$, and so $|Y(T)| =
|Y_1(T)| - O(n^{4.2/5})$. 
\end{proof}

Fix for the rest of the section the set $Y(T)$ that is guaranteed to exist by
Claim~\ref{claim:7:2:1}. For brevity, we mark a few properties of $Y(T)$ for
future reference. Let~(P1) be the property that every triangle $G \in Y(T)$ has
at least two edges in $\Trav_i \cup \BIGBite_{i+1}$ and that every triple
$(G_1, G_2, G_3) \in Y(T)$ is such that the triangle $G_1 \cup G_2 \cup G_3$
has at least two edges in $\Trav_i \cup \BIGBite_{i+1}$.  Let~(P2) be the
property that for every triangle $G$ there are at most $3$ possible triples
$(G_1, G_2, G_3) \in Y(T)$ such that $G = G_1 \cup G_2 \cup G_3$.  Let~(P3) be
the property that every edge in $\NotTrav_i \setminus \BIGBite_{i+1}$ belongs
to at most $3$ triangles and triples in $Y(T)$ (where we say that an edge
belongs to a triple $(G_1, G_2, G_3)$ if that edge belongs to $G_1 \cup G_2
\cup G_3$).  Using the fact that $Y(T) \subseteq Y_1(T)$, where $Y_1(T)$ is as
defined in Claim~\ref{claim:7:2:1}, and using Claim~\ref{claim:7:2:1}, we find
that the properties~(P1),~(P2)~and~(P3) hold.

Let $c \in \epsilon_3^2 \epsilon_2^{-1} \pm 1$ be an odd integer.  Let
$\sum_{G}$ range over all triangles $G \in Y(T)$ and let $\sum_{(G_1, G_2,
G_3)}$ range over all triples $(G_1, G_2, G_3) \in Y(T)$.  Let
\begin{eqnarray*}
W_5 &\deq& \sum_{G} \ONE[\I_c(\emptyset, G \cap \NotTrav_i, R) \wedge
\S_c(\emptyset, G \cap \NotTrav_i, R) \wedge |G \cap \Bite_{i+1}| = 0 ] + \\
&& \sum_{(G_1, G_2, G_3)}
\ONE[\I_c(G_2, G_3, R) \wedge \S_c(G_2, G_3, R) \wedge G_2 \subseteq
\Bite_{i+1}  \wedge |G_3 \cap \Bite_{i+1}|=0 ].
\end{eqnarray*}
By Lemma~\ref{lemma:conn2}, there are at least $W_5$ triangles $G$ such that
$|\M_{i+1} \cap G| = 2$, $|\NotTrav_{i+1} \cap G| = 1$, and letting $g$ denote
the edge in $\NotTrav_{i+1} \cap G$, one of the following two possibilities
hold.  Either $|X_{i+1,0}(g)| = 0$, in which case $G \in Y_{i+1,1}(T)$, or
$|X_{i+1,0}(g)| > 0$ and for every $G_0 \in X_{i+1,0}(g)$, $G_0$ shares at
least three vertices with $R$. 
The number of triangles $G$ for which the second possibility holds is at most
$|Z_{i+1}(R,T)| - |Z_i(R,T)|$. Therefore, there are at least $W_5 -
|Z_{i+1}(R,T)| + |Z_i(R,T)|$ triangles $G$ for which the first possibility
holds, or in other words, $|Y_{i+1,1}(T)| \ge W_5 - |Z_{i+1}(R,T)| +
|Z_i(R,T)|$.  This, with the next two claims, gives~(\ref{eq:sec:7:3}) for
$j=1$.

\begin{claim}  \label{claim:7:2:2}
$\expec(W_5) \ge y_{i+1,1,t} (1 - 100\Gamma_i - 99 \Gamma_i \gamma_i) - |Z_i(R,
T)|$.
\end{claim}
\begin{proof} 
Consider a triangle $G \in Y(T)$ and a triple $(G_1, G_2, G_3) \in Y(T)$.
Since an edge in $\BIGBite_{i+1}$ is an edge in $\Bite_{i+1}$ with probability
$n^{-\epsilon_1 - \epsilon_3} / (1 - i n^{-\epsilon_1 - \epsilon_2})$, using
Lemma~\ref{fact:f1}, we find that $\prob(|G \cap \Bite_{i+1}| = 0) \ge 1 -
o(\Gamma_i \gamma_i)$ and $\prob(G_2 \subseteq \Bite_{i+1} \wedge |G_3 \cap
\Bite_{i+1}| = 0) \ge n^{-\epsilon_1|G_2| - \epsilon_3|G_2|}(1 - o(\Gamma_i
\gamma_i))$.  Hence, by Lemma~\ref{lemma:conn2} and
Lemma~\ref{lemma:survive:process},
\begin{eqnarray*}
\prob(\I_c(\emptyset, G \cap \NotTrav_i, R) \wedge S_c(\emptyset, G \cap
\NotTrav_i, R) \wedge |G \cap \Bite_{i+1}| = 0) &\ge& \\
\bigg( \frac{\phi((i+1)n^{-\epsilon_1})} {\phi(i n^{-\epsilon_1})} \bigg) (1 -
91 \Gamma_i \gamma_i)
\end{eqnarray*}
and 
\begin{eqnarray*}
\prob(\I_c(G_2, G_3, R) \wedge \S_c(G_2, G_3, R) \wedge G_2 \subseteq
\Bite_{i+1} \wedge |G_3 \cap \Bite_{i+1}| = 0) &\ge& \\ 
n^{-\epsilon_1 |G_2| - \epsilon_3 |G_2|} \bigg(
\frac{\Phi((i+1)n^{-\epsilon_1}) - \Phi(in^{-\epsilon_1})}{n^{-\epsilon_1}
\phi(in^{-\epsilon_1})} \bigg)^{|G_2|} \bigg( \frac{\phi((i+1)n^{-\epsilon_1})}
{\phi(i n^{-\epsilon_1})} \bigg)^{|G_3|} (1 - 91 \Gamma_i \gamma_i).
\end{eqnarray*}
Also, from~(B3),~(D3) and Claim~\ref{claim:7:2:1}, it follows that the number
of triangles in $Y(T)$ is at least $y_{i,1,t} (1 - 100 \Gamma_i - 2 \Gamma_i
\gamma_i) - |Z_i(R, T)|$, the number of triples $(G_1, G_2, G_3)$ in $Y(T)$
which belong to $Y'_{i,2,1}(T)$ is at least $2 n^{\epsilon_3 - 2/5} y_{i,2,t}
(1 - 100 \Gamma_i - 2 \Gamma_i \gamma_i)$, and the number of triples $(G_1,
G_2, G_3)$ in $Y(T)$ which belong to $Y'_{i,3,2}(T)$ is at least $3
n^{2(\epsilon_3 - 2/5)} y_{i,3,t} (1 - 100\Gamma_i - 2 \Gamma_i \gamma_i)$.
The claim now follows using linearity of expectation.
\end{proof}

\begin{claim}
The probability (over the choice of $\BigBite_{i+1}$, $\Bite_{i+1}$, and the
choice of the birthtimes of the edges in $\Bite_{i+1}$) that $W_5$ deviates
from its expectation by more than $\Gamma_i \gamma_i y_{i+1,1,t}$ is at most
$n^{-\omega(s)}$.
\end{claim}
\begin{proof}
For an edge $f \in \BIGBite_{i+1}$, let $\Triangles(f)$ be the set which
contains every triangle $G \in Y(T)$ for which it holds that $f$ has the
potential of being a label in a tree in the forest $\{\TT_c(g,R) : g \in G \cap
\NotTrav_i\}$, and every triple $(G_1, G_2, G_3) \in Y(T)$ for which it holds
that $f$ has the potential of being a label in a tree in the forest $\{\TT_c(g)
: g \in G_2\} \cup \{ \TT_c(g,R) : g \in G_3\}$.
Observe that changing the outcome of an edge $f \in \BIGBite_{i+1}$ can change
$W_5$ by at most an additive factor of $|\Triangles(f)|$.

We now use~(P1),~(P2)~and~(P3), together with~(\ref{eq:roots1}),
(\ref{eq:roots2}) and~(\ref{eq:maxdeg}), to find that for every $f \in
\BIGBite_{i+1}$,
\begin{eqnarray*}
|\Triangles(f)| \le 3 \cdot |\Roots_c(f)| + 3 \cdot |\Roots_c(f) \cap
\BIGBite_{i+1}| \cdot 2 n^{1.1/5} \le 4 n^{2.1/5}.
\end{eqnarray*}
Moreover, by~(\ref{eq:labels}) we have that for every triangle $G \in Y(T)$,
there are at most $n^{0.01}$ edges $f \in \BIGBite_{i+1}$ such that $G \in
\Triangles(f)$, and likewise for every triple $(G_1, G_2, G_3) \in Y(T)$, there
are at most $3 n^{0.01}$ edges $f \in \BIGBite_{i+1}$ such that $(G_1, G_2,
G_3) \in \Triangles(f)$, and so
\begin{eqnarray*}
\sum_{f \in \BIGBite_{i+1}} |\Triangles(f)| \le |Y(T)| \cdot 3n^{0.01} \le
n^{5.1/5},
\end{eqnarray*}
where the second inequality follows from the definition of $Y(T)$ and
from~(\ref{eq:7:2:1}).  

We want to claim that $\sum_{f \in \BIGBite_{i+1}} |\Triangles(f)|^2 =
O(n^{6.95/5})$.  This will allow us to complete the proof using McDiarmid's
inequality, as we did in the previous cases.  However, unlike the situation in
the previous cases, the discussion above does not provide directly such a bound
on that sum of squares. We would have to resort to a finer analysis. The first
thing to note here is  that if $\sum_{f}$ ranges over all edges $f \in
\BIGBite_{i+1}$ with $|\Triangles(f)| \le n^{1.85/5}$, then using the
discussion above,
\begin{eqnarray*}
\sum_{f} |\Triangles(f)|^2 \le n^{1.85/5} \cdot \sum_{f \in \BIGBite_{i+1}}
|\Triangles(f)| 
\le n^{6.95/5}.
\end{eqnarray*}
Second, and this we show below, the number of edges $f \in \BIGBite_{i+1}$ with
$|\Triangles(f)| \ge n^{1.85/5}$ is at most $n^{2.45/5}$. Therefore,
\begin{eqnarray*}
\sum_{f \in \BIGBite_{i+1}} |\Triangles(f)|^2 \le n^{6.95/5} + n^{2.45/5} \cdot
(4 n^{2.1/5})^2 = O( n^{6.95/5} ).
\end{eqnarray*}
It now follows from McDiarmid's inequality that the probability that $W_5$
deviates from its expectation by more than $\Gamma_i \gamma_i y_{i+1,1,t} \ge
n^{4.99/5}$ is at most $n^{-\omega(s)}$.

To finish the proof, we argue that the number of edges $f \in \BIGBite_{i+1}$
with $|\Triangles(f)| \ge n^{1.85/5}$ is at most $n^{2.45/5}$. To this end,
assume for the sake of contradiction that this does not hold, and fix a set
$E_1$ of $\Theta(n^{2.45/5})$ edges $f \in \BIGBite_{i+1}$ with
$|\Triangles(f)| \ge n^{1.85/5}$.

For an edge $f \in \BIGBite_{i+1}$, let $\Edges_{1}(f)$ be the set of edges $g
\in \binom{R}{2} \cap (\NotTrav_i \setminus \BIGBite_{i+1})$ for which it holds
that $f$ has the potential of being a label in the tree $\TT_c(g, R)$ and $g$
belongs to some triangle or triple in $\Triangles(f)$. Furthermore, let
$\Edges_{2}(f)$ be the set of edges $g \in \binom{R}{2} \cap \BIGBite_{i+1}$
for which it holds that $f$ has the potential of being a label in the tree
$\TT_c(g)$ and $g$ belongs to some triangle or triple in $\Triangles(f)$.
Note that by~(P3) and since $\Triangles(f) \subseteq Y(T)$, every edge $g \in
\NotTrav_i \setminus \BIGBite_{i+1}$ belongs to at most $3$ triangles and
triples in $\Triangles(f)$.
Also, note that by~(P1),~(P2),~(\ref{eq:maxdeg}) and since $\Triangles(f)
\subseteq Y(T)$, every edge $g \in \BIGBite_{i+1}$ belongs to at most $3 \cdot
2n^{1.1/5}$ triangles and triples in $\Triangles(f)$.
Hence, for every edge $f \in \BIGBite_{i+1}$, by the definition of
$\Triangles(f)$, $\Edges_1(f)$ and $\Edges_2(f)$, we find that 
\begin{eqnarray*}
|\Triangles(f)|  \le 3 \cdot |\Edges_{1}(f)| + 3 \cdot |\Edges_{2}(f)| \cdot 2
n^{1.1/5}.
\end{eqnarray*}
In addition, for an edge $f \in \BIGBite_{i+1}$, by~(\ref{eq:roots1}),
\begin{eqnarray*}
|\Edges_{2}(f)| \cdot n^{1.1/5} \le |\Roots_c(f) \cap \BIGBite_{i+1}| \cdot
n^{1.1/5} \le n^{0.01} \cdot n^{1.1/5} = o(n^{1.85/5}),
\end{eqnarray*}
and so for every edge $f \in E_1$, 
\begin{eqnarray*}
3 \cdot | \Edges_1(f) | \ge |\Triangles(f)| - o(n^{1.85/5}) \ge 0.5 n^{1.85/5}.
\end{eqnarray*}
We will reach a contradiction by showing that for some $f \in E_1$,
$|\Edges_1(f)| = o(n^{1.85/5})$.

Define
\begin{eqnarray*}
E_2 \deq \bigcup_{f \in E_1} \Edges_1(f).
\end{eqnarray*}
Observe that for an edge $f \in E_1$ and for an edge $g \in \Edges_1(f)$,
there exists an edge $e \in \BIGBite_{i+1}$ for which the
following two properties hold: first, $f$ has the potential of being a label in
the tree $\TT_c(e)$; second, there exists a path of length two in $(\Trav_i
\cup \BIGBite_{i+1}) \cap \binom{R}{2}$ that completes $g$ to a triangle, and a
graph $G \in X_{0,5}(g)$ with $G \subseteq (\Trav_i \cup \BIGBite_{i+1})
\setminus \binom{R}{2}$ and $e \in G$ (and so, in particular, $e \in
\binom{[n]}{2} \setminus \binom{R}{2}$).
Hence, by the definition of $E_2$, there exists a set $E$
of edges in $\BIGBite_{i+1}$ for which the following two properties hold:
first, every edge $f \in E_1$ has the potential of being a label in a tree
$\TT_c(e)$ for some $e \in E$;  second, for every edge $g \in E_2$ there exists
a path of length two in $(\Trav_i \cup \BIGBite_{i+1}) \cap \binom{R}{2}$ that
completes $g$ to a triangle, and a graph $G \in X_{0,5}(g)$, with $G \subseteq
(\Trav_i \cup \BIGBite_{i+1}) \setminus \binom{R}{2}$ and $G \cap E \ne
\emptyset$ (and so, in particular, $E \subseteq \binom{[n]}{2} \setminus
\binom{R}{2}$).
By~(\ref{eq:roots1}), $|E| \le |E_1| \cdot n^{0.01} = O(n^{1/2})$. Hence,
by~(C8), $|E_2| = O( n^{4.2/5} )$. Thus, by~(\ref{eq:labels}), 
\begin{eqnarray*}
\sum_{f \in E_1} |\Edges_1(f)| \le |E_2| \cdot n^{0.01} = O(n^{4.25/5}).
\end{eqnarray*}
Therefore, there exists an edge $f \in E_1$ for which $|\Edges_1(f)| =
O(n^{4.25/5}) / |E_1|$. Since by assumption $|E_1| = \Theta(n^{2.45/5})$, we
get that there exists an edge $f \in E_1$ for which $|\Edges_1(f)| =
o(n^{1.85/5})$.  This gives the desired contradiction.
\end{proof}

\section{Proof of Theorem~\ref{thm:main2}} \label{sec:8}
%
We begin with the following definitions.  Let $0 \le i < I$.  For an integer $1
\le j \le 5$ and an edge $f \in \NotTrav_i$, let $X'''_{i,j}(f)$ be the set of
all $G \in X_{i,j}(f)$ such that $G \subseteq \M_i \cup \Bite_{i+1}$.  Consider
the undirected graph whose vertex set is the family of all edges in
$\Bite_{i+1}$, and whose edge set is the family of all sets $\{g_1, g_2\}$ such
that $g_2 \in G$ for some $G \in \bigcup_{1 \le j \le 5} X'''_{i,j}(g_1)$ (or
equivalently, $g_1 \in G$ for some $G \in \bigcup_{1 \le j \le 5}
X'''_{i,j}(g_2)$).  Let $\F_i$ be the event that the size of the largest
connected component in that graph has size $O_{\epsilon_1}(1)$ (where the
subscript $\epsilon_1$ means, as usual, that the hidden constant depends on
$\epsilon_1$).

To understand the motivation behind the above definitions, we make the
following observation. Let $0 \le i < I$. Assume that we are given $\M_i$,
$\Bite_{i+1}$, and a set $\Triangles$ of triangles, each having two edges in
$\M_i$ and one edge in $\Bite_{i+1}$ which can be added to $\M_i$ without
creating a copy of $K_4$. Further assume that for every triangle $G \in
\Triangles$, the edge in $G \cap \Bite_{i+1}$ belongs to exactly one triangle
in $\Triangles$.  Lastly assume that $\F_i$ holds.
Then the following three properties hold: first, the event that a triangle in
$\Triangles$ is contained in $\M_{i+1}$ depends only on the birthtimes of
$O_{\epsilon_1}(1)$ edges in $\Bite_{i+1}$; second, and this follows from the
first property, given any triangle in $\Triangles$, the probability that this
triangle is contained in $\M_{i+1}$ is $\Omega_{\epsilon_1}(1)$; third,
changing the birthtime of a single edge in $\Bite_{i+1}$ can change the number
of triangles in $\Triangles$ that are contained in $\M_{i+1}$ by at most an
additive factor of $O_{\epsilon_1}(1)$.  Later, this observation will be used
to prove Theorem~\ref{thm:main2}. For now, we prove the following.
\begin{lemma} \label{lemma:sec:8:1}
For $0 \le i < I$, $\prob(\F_i) \ge 1 - n^{-1}$.
\end{lemma}
\begin{proof}
Fix $0 \le i < I$.  An $(f,m)$-cluster is a sequence $(G_l)_{l=1}^m$ such that
for all $1 \le l \le m$ the following holds: $G_l \in X_{0,5}(g)$ for some $g
\in \{f\} \cup \bigcup_{k<l} G_k$, and $G_l$ shares at most four edges with
$\{f\} \cup \bigcup_{k<l} G_k$.
Say that an $(f,m)$-cluster $(G_l)_{l=1}^m$ is bad, if for all $1 \le l \le m$,
$G_l \subseteq \Trav_i \cup \Bite_{i+1}$ and $|G_l \cap \Bite_{i+1}| \ge 1$.
It should be clear that if there exists an integer $m$ such that for every $f
\in \NotTrav_i$ there is no bad $(f,m)$-cluster, then the largest connected
component in the graph the underlies the definition of the event $\F_i$ has
size at most $6m$.  Thus, by the union bound and by Markov's inequality, it
suffices to prove that for a fixed $f \in \NotTrav_i$ and for some $m =
O_{\epsilon_1}(1)$, the expected number of bad $(f,m)$-clusters is at most
$n^{-4}$.

Fix an edge $f \in \NotTrav_i$.  For an $(f,m)$-cluster $(G_l)_{l=1}^m$, we say
that $G_l$ is a $j$-type if $G_l$ shares exactly $j$ edges with $\{f\} \cup
\bigcup_{k<l} G_k$. We further say that $(G_l)_{l=1}^m$ is an $(a_0, a_1, a_2,
a_3, a_4)$-type if the number of graphs $G_l$ of $j$-type is $a_j$.
Fix a sufficiently large integer $m = O_{\epsilon_1}(1)$.  Fix a set of
integers $\{a_j : 0 \le j \le 4\}$, so that $\sum_{j=0}^4 a_j = m$. Let $a$ be
the vector $(a_0, a_1, a_2, a_3, a_4)$.  The number of $a$-type
$(f,m)$-clusters of length $m$ is trivially at most $O_{\epsilon_1}( n^{2a_0 +
a_1 + a_2})$.  Note that for any edge $g \in \binom{[n]}{2}$, we have $\prob(g
\in \Trav_i \cup \Bite_{i+1}) \le q_1 \deq 2 n^{\epsilon_1^2 - 2/5}$ and
$\prob(g \in \Bite_{i+1}) \le q_2 \deq 2 n^{-\epsilon_1 - 2/5}$.  Therefore,
the probability that an $a$-type $(f,m)$-cluster is bad is at most $(q_1^4
q_2)^{a_0} \cdot q_1^{4a_1 + 3a_2 + 2a_3 + a_4}$.  It follows that the expected
number of bad $a$-type $(f, m)$-clusters is at most $n^{-\Omega(\epsilon_1
m)}$, which is at most $n^{-5}$, as $m$ is sufficiently large.  Since there are
at most $O_{\epsilon_1}(1)$ choices for the vector $a$, a union bound argument
completes the proof.
\end{proof}

We turn to prove Theorem~\ref{thm:main2}.  We need to show that a.a.s., every
set $S \subseteq [n]$ of $s$ vertices of $\M_I$ spans a triangle. In other
words, letting $\exists S$ stand for ``there exists a set $S \subseteq [n]$ of
$s$ vertices,'' and letting $K_3 \nsubseteq \M_i \cap \binom{S}{2}$ stand for
``$\M_i \cap \binom{S}{2}$ is triangle-free,'' we need to show that
$\prob(\exists S : K_3 \nsubseteq \M_I \cap \binom{S}{2}) = o(1)$.
Say that the process behaves if for every $0 \le i < I$, $\A_i \wedge \B_i
\wedge \C_i \wedge \F_i$ holds and $\M_i$ has maximum degree at most $0.01s$. From Lemma~\ref{lemma0},
Lemma~\ref{lemma:sec:8:1} and a result of Bohman and Keevash~\cite[Theorem~1.6]{BKeevash} it follows that $\prob(\text{process behaves}) = 1 -
o(1)$. Therefore, 
\begin{eqnarray*}
\prob(\exists S: K_3 \nsubseteq \M_I \cap \tbinom{S}{2}) &\le& \prob(\exists S:
K_3 \nsubseteq \M_I \cap \tbinom{S}{2} \given \text{process behaves}) +
\prob(\neg (\text{process behaves})) \\
&=& \prob(\exists S: K_3 \nsubseteq \M_I \cap \tbinom{S}{2} \given
\text{process behaves}) + o(1).
\end{eqnarray*}
Thus, it remains to show that $\prob(\exists S : K_3 \nsubseteq \M_I \cap
\binom{S}{2} \given \text{process behaves}) = o(1)$, and so by the union bound,
it remains to fix a set $S \subseteq [n]$ of $s$ vertices, and show that
\begin{eqnarray*}
\prob(K_3 \nsubseteq \M_I \cap \tbinom{S}{2} \given \text{process behaves}) =
o(n^{-s}).
\end{eqnarray*}
For that, we use the following lemma, whose proof is given below.
\begin{lemma} \label{lemma:sec:8:2}
For $1 \le i < I$, 
\begin{eqnarray*}
\prob( K_3 \nsubseteq \M_{i+1} \cap \tbinom{S}{2} \given \text{process behaves}
\wedge K_3 \nsubseteq \M_i \cap \tbinom{S}{2}) \le \exp\Big(-
\Omega_{\epsilon_1} \Big( n^{-\epsilon_1 - 2/5} y_{i,1,s^3} \Big)\Big).
\end{eqnarray*}
\end{lemma}
Let $I_0 \deq \floor{ n^{\epsilon_1 + 0.5\epsilon_1^2} }$ and recall that $I =
\floor{ n^{\epsilon_1 + \epsilon_1^2} }$.  By Lemma~\ref{lemma:sec:8:2},
\begin{eqnarray*}
\prob(K_3 \nsubseteq \M_I \cap \tbinom{S}{2} \given \text{process behaves})
&\le& \prod_{1 \le i < I} \exp\Big( - \Omega_{\epsilon_1}\Big( n^{-\epsilon_1 -
2/5} y_{i,1,s^3} \Big) \Big) \\ 
&\le& \prod_{I_0 \le i < I} \exp\Big( - \Omega_{\epsilon_1}\Big( n^{-\epsilon_1
- 2/5} y_{i,1,s^3} \Big) \Big) \\
& = & \exp\Big( -\Omega_{\epsilon_1}(1) \cdot \sum_{I_0 \le i < I}
n^{-\epsilon_1 - 2/5} y_{i,1,s^3} \Big).
\end{eqnarray*}
Also, for every $I_0 \le i < I$, by the definition of $y_{i,1,s^3}$ and by
Lemma~\ref{fact:f1},
\begin{eqnarray*}
n^{-\epsilon_1 - 2/5} y_{i,1,s^3} 
\ge
s^3 n^{-\epsilon_1 -6/5} \Phi(in^{-\epsilon_1})^2 \phi(in^{-\epsilon_1})
= \Omega\bigg(  \frac{C^2 s}{i} \bigg).
\end{eqnarray*}
Therefore, 
\begin{eqnarray*}
\prob(K_3 \nsubseteq \M_I \cap \tbinom{S}{2} \given \text{process behaves})
&\le& \exp\Big( -\Omega_{\epsilon_1}(1) \cdot C^2 s \cdot \sum_{I_0 \le i < I}
\frac{1}{i} \Big) \\ 
&\le& \exp\Big( -\Omega_{\epsilon_1}(1) \cdot C^2 s \ln n \Big).
\end{eqnarray*}
Taking $C = C(\epsilon_1)$ sufficiently large, the last bound is at most
$o(n^{-s})$. This gives Theorem~\ref{thm:main2}.

\subsection{Proof of Lemma~\ref{lemma:sec:8:2}}
%
Fix $1 \le i < I$. We want to bound the probability that $\M_{i+1} \cap
\binom{S}{2}$ is triangle-free, conditioned on the event that the process
behaves and that $\M_i \cap \binom{S}{2}$ is triangle-free. For that purpose,
assume that we are given $\M_i$ so that 
$\A_i \wedge \B_i \wedge \C_{i-1}$ holds, $\M_i$ has maximum degree  at most
$0.01s$, and $\M_i \cap \binom{S}{2}$ is triangle-free.  Given that
assumption, we bound the probability that $\M_{i+1} \cap \binom{S}{2}$ is
triangle-free (where the probability is over the choice of $\BIGBite_{i+1},
\BigBite_{i+1}, \Bite_{i+1}$ and the choice of the birthtimes of the edges in
$\Bite_{i+1}$), conditioned on the event that the process behaves.

Let $S_i \subseteq S$ be the set that is guaranteed to exist by $\B_i$. 
\begin{claim} \label{claim:yi1}
There exists a pair $(R, T) \in \Pairs(S_i)$ such that
$|Y_{i,1}(T)| = \Omega( y_{i,1,s^3} )$.
\end{claim}
\begin{proof}
By~(B3) and Lemma~\ref{fact:f1}, for every pair $(R, T) \in \Pairs(S_i)$ with
$|T| = t$, we have $|Y_{i,1}(T)|  = \Omega( y_{i,1,t} ) - |Z_i(R, T)|$, and
furthermore, since $t = \Omega(s^3)$, we have $y_{i,1,t} =
\Omega(y_{i,1,s^3})$. Hence, for every pair $(R,T) \in \Pairs(S_i)$, 
\begin{eqnarray*}
|Y_{i,1}(T)| = \Omega( y_{i,1,s^3} ) - |Z_i(R, T)|.
\end{eqnarray*}
We show below that for some pair $(R, T) \in \Pairs(S_i)$, we have $|Z_i(R, T)|
= O(n^{4.2/5})$. This, together with Lemma~\ref{fact:f1}, will imply that
$|Z_i(R,T)| = o(y_{i,1,s^3})$ and so the claim will follow.

Consider an arbitrary pair $(R,T) \in \Pairs(S_i)$. Recall that $Z_i(R,T)$ is
the set of all triangles $G \in T$ such that $|\M_i \cap G| = 2$, $|\NotTrav_i
\cap G| = 1$, and letting $g$ denote the edge in $\NotTrav_i \cap G$, there
exists $G_0 \in X_{i,0}(g)$ such that $G_0$ shares at least three vertices with
$R$.
Since $\M_i \subseteq \Trav_{i-1} \cup \BIGBite_i$, since $\C_{i-1}$(C6) holds
and since $R \subseteq S$, there is a set $E_0$ of at most $n^{3.1/5}$ edges,
the removal of which from $\M_i$ leaves at most $n^{4.2/5}$ $4$-cycles in $\M_i
\cap \binom{R}{2}$.  In addition, since $\M_i \subseteq \Trav_i$, since~(B2)
holds and since $R \subseteq S_i$, every edge in $E_0$ belongs to at most
$2n^{1.1/5}$ triangles in $Z_i(R,T)$.
Hence, except for at most $4n^{4.2/5} + |E_0| \cdot 2n^{1.1/5} = O(n^{4.2/5})$
triangles, for every triangle $G \in Z_i(R,T)$, we have that the edge $g \in
\NotTrav_i \cap G$ belongs to no other triangle in $Z_i(R,T)$.
Therefore, up to an additive factor of $O(n^{4.2/5})$, the number of triangles
in $Z_i(R,T)$ is at most the number of edges $g \in \binom{R}{2}$ that belong
to some triangle in $T$, and for which there exists $G_0 \in X_{i,0}(g)$ such
that $G_0$ shares at least three vertices with $R$.

Since $\M_i \subseteq \Trav_{i-1} \cup \BIGBite_i$, since $\C_{i-1}$(C7)
holds and since $s - o(s) \le |S_i| \le |S| \le s$ by~(B1), assuming the 
maximum degree in $\M_i \cap \binom{S_i}{2}$ is at most $n^{1.1/5}$,
$S_i$ satisfies the following two properties. First, there are at most
$O(n^{4.2/5})$ edges $g \in \binom{S_i}{2}$ for which there exists a graph $G_0
\in X_{i,0}(g)$, which shares all four vertices with $S_i$.  Second, there is a
set $R_0 \subseteq [n] \setminus S_i$ of at most $n^{0.99/5}$ vertices (which
we fix for the rest of the proof), such that there are at most $O(n^{4.2/5})$
edges $g \in \binom{S_i}{2}$ for which there exists a graph $G_0 \in
X_{i,0}(g)$, which shares exactly three vertices with $S_i$ and one vertex with
$[n] \setminus (S_i \cup R_0)$.
Also note that since $\M_i \subseteq \Trav_i$ and since~(B2) holds, 
the maximum degree in $\M_i \cap \binom{S_i}{2}$ is at most $n^{1.1/5}$.
Hence, it remains to show that for some pair $(R,T) \in \Pairs(S_i)$, the
number of edges $g \in \binom{R}{2}$ that belong to some triangle in $T$, and
for which there exists $G_0 \in X_{i,0}(g)$ such that $G_0$ shares three
vertices with $R$ and one vertex with $R_0$ is at most $O(n^{4.2/5})$.

Let $V \subseteq S_i$ be the set of vertices $v \in S_i$ such that $v$ is
adjacent in $\M_i$ to at most one vertex in $R_0$. 
We have the following two observations. The first observation is that since
$\M_i \subseteq \Trav_{i-1} \cup \BIGBite_i$ and since $\C_{i-1}$(C3) holds,
$|S_i \setminus V| \le |R_0|^2 \cdot n^{1/5 + 10\epsilon_3} = o(s)$, and so
$|V| \ge s - o(s)$. 
The second observation is that by this lower bound on the size of $V$ and
since $\M_i$ has maximum degree at most $0.01s$,
there exists a partition of the vertices of $V$ to three sets of
vertices, each of size $\Omega(s)$, such that there is no vertex in $R_0$ that
is adjacent in $\M_i$ to two vertices in two different parts of the partition.
Fix such a partition and consider the pair $(V, T) \in \Pairs(S_i)$ that
corresponds to that partition.
Then there is no triangle in $T$ with an edge whose two vertices
are adjacent in $\M_i$ to a vertex in $R_0$.
Hence, the number of edges $g \in \binom{V}{2}$ that belong to some triangle in
$T$, and for which there exists $G_0 \in X_{i,0}(g)$ such that $G_0$ shares
three vertices with $V$ and one vertex with $R_0$ is $0$.  This completes the
proof.
\end{proof}

Fix for the rest of the proof a pair $(R,T) \in \Pairs(S_i)$, as guaranteed to
exist by the above claim, so that $Y_{i,1}(T) = \Omega(y_{i,1,s^3})$.
Let $\E$ be the event that there exists a set $\Triangles \subseteq
Y_{i,1}(T)$ with the following three properties: first, every triangle in
$\Triangles$ has two edges in $\M_i$ and one edge in $\Bite_{i+1}$ (which can
be added to $\M_i$ without creating a copy of $K_4$); second, for every
triangle $G \in \Triangles$, the edge in $G \cap \Bite_{i+1}$ belongs to
exactly one triangle in $\Triangles$; third, $|\Triangles| = \Omega(
n^{-\epsilon_1 - 2/5} y_{i,1,s^3} )$.  We have
\begin{eqnarray*}
\prob( K_3 \nsubseteq \M_{i+1} \cap \tbinom{S}{2} \given \text{process behaves}
\wedge K_3 \nsubseteq \M_i \cap \tbinom{S}{2}) &\le& \\ 
\prob( K_3 \nsubseteq \M_{i+1} \cap \tbinom{S}{2} \given \E \wedge
\text{process behaves} \wedge K_3 \nsubseteq \M_i \cap \tbinom{S}{2}) &+& \\
\prob(\neg \E \given \text{process behaves} \wedge K_3 \nsubseteq \M_i \cap
\tbinom{S}{2} ).
\end{eqnarray*}
We bound each of the two last terms by $\exp( - \Omega_{\epsilon_1}( n^{-
\epsilon_1 - 2/5} y_{i,1,s^3} ))$.  

To bound the first term,  it is enough to bound the probability of the event
$K_3 \nsubseteq \M_{i+1} \cap \binom{S}{2}$, under the assumption that we are
given $\M_i$ and $\Bite_{i+1}$, $\E$ holds, the process behaves, and $K_3
\nsubseteq \M_i \cap \binom{S}{2}$. (The probability here is over the choice of
the birthtimes of the edges in $\Bite_{i+1}$.)
Under this assumption, we can use the observation that was given at the
beginning of the section to claim the following.  First, the event that a
triangle in $\Triangles$ is contained in $\M_{i+1}$ depends only on the
birthtimes of $O_{\epsilon_1}(1)$ edges in $\Bite_{i+1}$.  Second, the expected
number of triangles in $\Triangles$ that are contained in $\M_{i+1}$ is
$\Omega_{\epsilon_1}(|\Triangles|)$, which is $\Omega_{\epsilon_1}( n^{-
\epsilon_1 - 2/5} y_{i,1,s^3})$.  Third, changing the birthtime of a single
edge in $\Bite_{i+1}$ can change the number of triangles in $\Triangles$ that
are contained in $\M_{i+1}$ by at most an additive factor of
$O_{\epsilon_1}(1)$.  Therefore, the probability of the event $K_3 \nsubseteq
\M_{i+1} \cap \binom{S}{2}$ is at most the probability that no triangle in
$\Triangles$ is in $\M_{i+1}$, which given the assumptions and the three claims
above, by McDiarmid's inequality, is at most
$\exp(-\Omega_{\epsilon_1}(n^{-\epsilon_1 - 2/5} y_{i,1,s^3}))$, as needed.

Next, we bound the second term. We claim that under the assumption that we are
given $\M_i$ so that the process behaves and $K_3 \nsubseteq \M_i \cap
\binom{S}{2}$ holds, $\neg \E$ occurs with probability at most $\exp ( -
\Omega(n^{-\epsilon_1 - 2/5} y_{i,1,s^3}) )$. (The probability here is over the
choice of $\BIGBite_{i+1}, \BigBite_{i+1}$ and $\Bite_{i+1}$.) 
To prove that claim, first recall that $|Y_{i,1}(T)| = \Omega(y_{i,1,s^3})$.
Next, using an argument similar to the one used in the proof of
Claim~\ref{claim:yi1}, one can find that there is a set $Y^*_{i,1}(T) \subseteq
Y_{i,1}(T)$ of size $\Omega(|Y_{i,1}(T)|)$, that is of size
$\Omega(y_{i,1,s^3})$, such that for every triangle $G \in Y^*_{i,1}(T)$, the
edge in $\NotTrav_i \cap G$ belongs to exactly one triangle in $Y^*_{i,1}(T)$.
In particular, every triangle $G \in Y^*_{i,1}(T)$ is uniquely determined by
the edge in $\NotTrav_i \cap G$.  Consider the set of $\Omega(y_{i,1,s^3})$
edges that determine the triangles in $Y^*_{i,1}(T)$, and note that by
Chernoff's bound, $\Omega(n^{-\epsilon_1 - 2/5} y_{i,1,s^3})$ edges of these
are in $\Bite_{i+1}$ with probability at least $1 - \exp(-\Omega_{\epsilon_1}(
n^{-\epsilon_1 - 2/5} y_{i,1,s^3}))$.  So with probability at least $1 -
\exp(-\Omega_{\epsilon_1}( n^{-\epsilon_1 - 2/5} y_{i,1,s^3}))$, there is a set
of $\Omega(n^{-\epsilon_1 - 2/5} y_{i,1,s^3})$ triangles in $Y^*_{i,1}(T)$,
each of which has two edges in $\M_i$ and one edge in $\Bite_{i+1}$. This
completes the proof.



\end{document}